\documentclass[letterpaper]{article}

\usepackage{amsmath}
\usepackage{amsfonts}
\usepackage{amsthm}
\usepackage{amssymb}
\usepackage{datetime}
\usepackage{color}
\usepackage{setspace}

\theoremstyle{plain}% default
\newtheorem{theorem}{Theorem}[section]
\newtheorem{lemma}[theorem]{Lemma}

\newtheorem{corollary}[theorem]{Corollary}
\newtheorem*{theorem*}{Theorem}
\theoremstyle{definition}
\newtheorem{definition}[theorem]{Definition}
\newtheorem{example}[theorem]{Example}
\newtheorem{fact}[theorem]{Fact}
\newtheorem{problem}[theorem]{Problem}
\newtheorem{remark}[theorem]{Remark}
\newtheorem{notation}[theorem]{Notation}

\numberwithin{equation}{section}

\newcommand{\all}{\hbox{for all}}

\newcommand{\bra}[2]{\langle#1,#2\rangle}

\newcommand{\bigcupn}{\bigcup\nolimits}
\newcommand{\Bra}[2]{\big\langle#1,#2\big\rangle}
\newcommand{\coinc}{{\rm coinc}}

\newcommand{\dbs}{^{**}}
\newcommand{\dcoinc}{{\rm dcoinc}}

\newcommand{\dom}{\hbox{\rm dom}}
\newcommand{\e}[1]{{e^{(#1)}}}
\newcommand{\eighth}{\ts\frac{1}{8}}
\newcommand{\epione}{{\,\mathop{\oplus}\nolimits_1\,}}
\newcommand{\epeone}{{\,\mathop{\oplus}\nolimits_1^e\,}}
\newcommand{\epitwo}{{\,\mathop{\oplus}\nolimits_2\,}}
\newcommand{\epetwo}{{\,\mathop{\oplus}\nolimits_2^e\,}}
\newcommand{\epi}{\hbox{\rm epi}}
\newcommand{\eps}{\varepsilon}
\newcommand{\ex}{\hbox{there exist}}

\newcommand{\F}{{\mathbb F}}
\newcommand{\fourth}{\ts\frac{1}{4}}

\newcommand{\half}{{\textstyle\frac{1}{2}}}
\newcommand{\overh}{{\overline h}}

\newcommand{\infn}{\inf\nolimits}
\newcommand{\intr}{\hbox{\rm int}}

\newcommand{\limn}{\lim\nolimits}

\newcommand{\lin}{\hbox{\rm lin}}
\newcommand{\lr}{\Longrightarrow}
\newcommand{\LT}{{\wt L}}

\newcommand{\minn}{\min\nolimits}
\newcommand{\mth}{{\textstyle\frac{1}{m}}}
\newcommand{\on}{\hbox{on}}
\newcommand{\PC}{{\cal PC}}
\newcommand{\PCLSC}{{\cal PCLSC}}

\newcommand{\qlr}{\quad\Longrightarrow\quad}
\newcommand{\qLt}{q_{\wt L}}

\newcommand{\quand}{\quad\hbox{and}\quad}
\newcommand{\rbar}{\,]{-}\infty,\infty]}
\newcommand{\RR}{\mathbb R}
\newcommand{\rl}{\Longleftarrow}
\newcommand{\rLt}{r_{\wt L}}

\newcommand{\st}{\hbox{such that}}

\newcommand{\supn}{\sup\nolimits}

\newcommand{\toto}{\rightrightarrows}
\newcommand{\trs}{^{***}}
\newcommand{\ts}{\textstyle}
\newcommand{\tsum}{\textstyle\sum}

\newcommand{\wh}{\widehat}
\newcommand{\wt}{\widetilde}

\newcommand{\Cor}{Corollary~\ref}
\newcommand{\Cors}{Corollaries~\ref}
\newcommand{\Def}{Definition~\ref}
\newcommand{\Ex}{Example~\ref}
\newcommand{\Lem}{Lemma~\ref}
\newcommand{\Lems}{Lemmas~\ref}

\newcommand{\Rem}{Remark~\ref}
\newcommand{\Sec}{Section~\ref}
\newcommand{\Secs}{Sections~\ref}
\newcommand{\Thm}{Theorem~\ref}
\newcommand{\Thms}{Theorems~\ref}

\title{A stand--alone analysis of quasidensity}
\author{
Stephen Simons
\thanks{
Department of Mathematics, University of California, Santa Barbara, CA\ 93106-3080, U.S.A.
Email: \texttt{stesim38@gmail.com}.}}

\date{}

\begin{document}
\maketitle
\noindent
%{\em This paper is dedicated to my good friend Juan Enrique Mart\'\i nez-Legaz, with fond memories of an earthshaking moment that we spent together.}


maximal monotonicity, sum theorem, subdifferential.

\begin{abstract}\noindent
In this paper we consider the ``quasidensity'' of a subset of the product of a
Banach space and its dual, and give a connection between quasidense sets and sets of ``type (NI)''.   We discuss ``coincidence sets'' of certain convex functions and prove two sum theorems for coincidence sets.   We obtain new results on the Fitzpatrick extension of a closed quasidense monotone multifunction.  The analysis in this paper is self-contained, and independent of previous work on ``Banach SN spaces''.    
\end{abstract}

\section{Introduction}
In this paper we suppose that $E$ is a nonzero real Banach space with dual $E^*$.   In \cite{PARTTWO}, we defined the {\em quasidensity} of a subset of $E \times E^*$.   This was actually a special case of the concept of the {\em $r_L$--density} of a subset of a {\em Banach SN space} that had been previously defined in \cite{PARTONE},  and the analysis in \cite{PARTTWO} was heavily dependent on \cite{PARTONE}.   The purpose of this paper is to give a development of the properties of quasidensity that is independent of \cite{PARTONE}.   This paper also contains many results that did not appear in \cite{PARTTWO}. 
\par
In \Sec{FENCHELsec}, we discuss proper convex functions on a Banach space and their Fenchel conjugates and biconjugates.   We also introduce the (well known) canonical map from a Banach space into its {\em bidual}, which we denote by $\wh{\ }$.   In \Thms{HKFthm} and \ref{HKF2thm} and \Lem{RSlem}, we discuss some subtler properties of proper convex functions that are not necessarily lower semicontinuous.   These subtler properties will be used in \Thm{THREEthm}. 
\par
In \Sec{EEsec}, we discuss Banach spaces of the form $E \times E^*$.   For this kind of Banach space, there is a (not so well known) canonical map from the space into its {\em dual}, which we denote by $L$ \big(see \eqref{Ldef}\big).   We define the {\em quasidensity} of  of a subset of $E \times E^*$ (or, equivalently, of a multifunction from $E$ into $E^*$) in\break \Def{QDdef}.   The definition of quasidensity does not require monotonicity, though there is a rich theory of the interaction of quasidensity and monotonicity which we will discuss in \Secs{MONsec}--\ref{SPECsec} --- the definition of {\em monotonicity} does not actually appear until \Sec{MONsec}.   \Lem{SLlem}, \Thm{NIthm} and  \Cors{NIcor} and \ref{THAcor} contain useful results on quasidensity without a monotonicity assumption.   In particular, \Thm{NIthm} says that $L$ ``preserves quasidensity'', and we establish in \Cor{NIcor}  that every quasidense set is of {\em type (NI)}, a concept that has been extensively studied over the past two decades.   We will return to this issue below.  We mention in \Ex{SUBex} that the subdifferential of a proper, convex, lower semicontinuous function on $E$ is quasidense.   This result is generalized in \cite{SW} to certain more general subdifferentials of nonconvex functions.
\par
In \Sec{RLsec}, we initiate the theory of the {\em coincidence sets} of certain convex functions.   The basic idea is that we consider a proper convex function, $f$, on $E \times E^*$ that dominates the canonical bilinear form, $q_L$, and the corresponding coincidence set is the set on which $f$ and $q_L$ coincide.   (The ``$q$'' in this notation stands for ``quadratic''.)   The main results in this section (and the pivotal results of this paper) are \Thm{FCthm}  (the primal condition for quasidensity),\break \Thm{FSTARthm} (the dual condition for quasidensity) and \Thm{THREEthm} (the theorem of the three functions).  As we observed above, the definition of monotonicity is not used explicitly before \Sec{MONsec}, but monotonicity is hiding below the surface because, as we shall see in \Lem{CONTlem}, coincidence sets are monotone.
\par
In \Sec{EPIsec}, we investigate the coincidence sets of the partial episums of a pair of convex functions.   This analysis will lead to the two sum theorems for quasidense maximally monotone multifunctions that we will establish in \Thm{STDthm} and  \Thm{STRthm}. 
\par
We start our explicit discussion of monotonicity in \Sec{MONsec}.   We prove in \Thm{RLMAXthm} that every closed, monotone quasidense multifunction is maximally monotone.   On the other hand, we give examples of varying degrees of abstraction in \Ex{TAILex} and \Thms{SMAXthm}, \ref{SFTthm}(b) and \ref{SPECTthm} of maximally monotone linear operators that are not quasidense.   The link between \Sec{RLsec} and \Sec{MONsec} is provided by \Lem{CONTlem}, in which we give a short proof of the result first established by Burachik--Svaiter and Penot that coincidence sets are monotone.    So suppose that $S\colon\ E \toto E^*$ is monotone and $G(S) \ne \emptyset$.   In \Def{THdef}, we define the function, $\theta_S\colon E^* \times E\dbs \to \rbar$, by adapting \Def{THAdef}.  The well known {\em Fitzpatrick function}, $\varphi_S$, is defined in  \Def{PHdef} by $\varphi_S = \theta_S \circ L$.   There is a short history of the Fitzpatrick function in \Rem{FDEFrem}.   Now let $S$ be maximally monotone.   Then we prove in \Thm{PHthm} that $S$ is quasidense if, and only if, ${\varphi_S}^* \ge \qLt$ on $E^* \times E\dbs$, and we prove in \Thm{THthm} that $S$ is quasidense if, and only if, $\theta_S \ge \qLt$ on $E^* \times E\dbs$.   These two results enable us to give two partial converses to \Thm{RLMAXthm} in \Cors{SURJcor} and \ref{CONVcor}, namely that {\em if $S$ is maximally monotone and surjective then $S$ is quasidense} and that {\em if $E$ is reflexive and $S$ is maximally monotone then $S$ is quasidense}.   \Thm{THthm} is particularly significant because it shows that a maximally monotone multifunction $S$ is quasidense exactly when it is of type (NI).      
\par
\par
In \Sec{DSUMSsec}, we prove the {\em Sum theorem with domain constraints} that was established in \cite{PARTONE}.   It is important to realize that we do not merely give sufficient conditions for a sum theorem for a pair of maximally monotone multifunctions to hold.   In fact, we prove that, under the given conditions, the sum of a pair of {\em closed, monotone and quasidense} multifunctions is again {\em closed, monotone and quasidense}.
\par
In \Sec{FITZEXTsec}, we discuss the {\em Fitzpatrick extension} of a closed, monotone and quasidense multifunction.   This will be needed for our    analysis of the {\em Sum theorem with range constraints} that will be the topic of \Sec{RSUMSsec}.   If $S\colon\ E \toto E^*$ is closed, monotone and quasidense then the Fitzpatrick extension, $S^\F\colon\ E^* \toto E\dbs$, of $S$ is defined formally in terms of ${\varphi_S}^*$ in \eqref{PHSTCRIT}, and we give two other characterization of $S^\F$ in \eqref{THCRIT}. We prove in \Thm{AFMAXthm} that $S^\F$ is maximally monotone, but we will see in \Ex{TAILex}, \Thms{SFTthm}(b) and \ref{SPECTthm} that it may fail to be quasidense. \big(It is observed in \Rem{GOSSrem}, that $(y^*,y\dbs) \in G(S^\F)$ exactly when $(y\dbs,y^*)$ is in the {\em Gossez extension of $G(S)$}\big).   $S^\F$ is defined in rather an abstract fashion, but we give a situation in \Thm{TSthm} in which we can give a more explicit description of $S^\F$. \Thm{TSthm} was obtained by analyzing some results of Bueno and Svaiter on {\em linear} multifunctions, which we will discuss in greater detail in \Sec{NONQDEXTsec}.   \Thm{TSthm} does {\em not} have any linearity assumptions, but \Thm{LINVWthm} is an application to linear maps.      
\par
In \Sec{RSUMSsec}, we prove the {\em Sum theorem with range constraints} that was first established in \cite{PARTONE}.
\par
In \Sec{ANOTHERsec}, we discuss a slight modification of an example due to Bueno and Svaiter of a non-quasidense maximally monotone skew linear operator from a subspace of $c_0$ into  $\ell_1$.   In \Sec{NONQDEXTsec} we discuss a procedure due to Bueno and Svaiter for constructing quasidense linear maps from a Banach space into its dual with a non-quasidense Fitzpatrick extension.   In \Sec{SPECsec}, we give a specific example of the construction of \Sec{NONQDEXTsec}, a map from $c_0$ into $\ell_1$.
\par
Given a maximally monotone multifunction, there are a number of conditions that are equivalent to its quasidensity.   Broadly speaking, they separate into two classes, depending on whether or not they use the bidual in their definition.
\par
Conditions that do not use the bidual include the {\em negative alignment condition} (see \cite[Theorem 11.6, p.\ 1045]{PARTONE}), two ``fuzzy'' criteria for quasidensity \big(in which an element of $E^*$ is replaced by a nonempty $w(E^*,E)$--compact convex subset of $E^*$, or an element of $E$ is replaced by a nonempty $w(E,E^*)$--compact convex subset of $E$ --- see \cite[Section 8, pp. 14--17]{PARTTWO}\big) and the {\em type (FP)} condition \big(see \cite[Section 10, pp. 20--22]{PARTTWO}\big).

\par
There are many classes of maximally monotone multifunctions coinciding with those of type (FP) in the literature that {\em do} require the bidual in their\break definitions.   We mention  {\em type (D)}, {\em dense type}, {\em type (ED)} and {\em Type (NI)}.   These equivalences have been known for some time.  See \cite[Introduction, pp.\ 6--7]{PARTTWO} for a discusion of these with references to the sources of these results.
\par
The bidual is not mentioned explicitly in the {\em statements} of \Thm{Dthm}, \Cor{SURJcor} or \Thm{STDthm}, but our {\em proofs} of all of these results ultimately depend on the bidual at one point or another.   This raises the fascinating question whether there are proofs of any of these results that do not depend on the bidual.  This seems to be quite a challenge.   Another similar challenge is to find a proof that does not depend on the bidual of the fact that a maximally monotone multifunction is quasidense if, and only if, it is of type (FP).  Of course, such a proof could not go through the equivalence of both of these classes of multifunctions with those of type (NI).
\par
It was proved in \cite[Theorem 11.9, pp.\ 1045--1046]{PARTONE}  that every closed,\break monotone quasidense multifunction is of {\em type (ANA)}.   It was also proved in \cite[Theorem 7.2, p.\ 14 and Theorem 8.5, pp.\ 16--17]{PARTTWO} that every closed, monotone quasidense multifunction is of {\em type (FPV)}, and {\em strongly maximal}.   These observations lead to the three interesting problems of finding maximally monotone multifunctions that fail to be in any of these three classes.
\par
The author would like to thank Orestes Bueno for a very interesting\break discussion, which led to the analysis that we present in \Thm{TSthm} and\break \Secs{NONQDEXTsec} and \ref{SPECsec}.   This discussion took place during the author's stay in the Erwin Schrodinger International Institute for Mathematics and Physics of the University of Vienna in January-February, 2019. The author would like to express his sincere appreciation to the Erwin Schrodinger Institute for their support.
\par
All vector spaces in this paper are {\em real}.
%:\Sec{FENCHELsec}
\section{Fenchel conjugates}\label{FENCHELsec}
We start off by introducing some Banach space notation.
%:   \Def{PCdef}
If $X$ is a nonzero Banach space and $f\colon\ X \to \rbar$, we write $\dom\,f$ for  $\big\{x \in X\colon\ f(x) \in \RR\big\}$.   $\dom\,f$ is the {\em effective domain} of $f$.   We say that $f$ is {\em proper} if $\dom\,f \ne \emptyset$.   We write $\PC(X)$ for the set of all proper convex functions from $X$ into $\rbar$ and $\PCLSC(X)$ for the set of all proper convex lower semicontinuous functions from $X$ into $\rbar$.   We write $X^*$ for the dual space of $X$ \big(with the pairing $\bra\cdot\cdot\colon X \times X^* \to \RR$\big).  If $f \in \PC(X)$ then, as usual, we define the {\em Fenchel conjugate}, $f^*$, of $f$ to be the function on $X^*$ given by
%:   \eqref{FSTAR}
\begin{equation}\label{FSTAR}
f^*(x^*) := \supn_{X}\big[{x^*} - f\big] = \supn_{x \in X}\big[\bra{x}{x^*} - f(x)\big].
\end{equation}
\par
$X\dbs$ stands for the bidual of $X$ \big(with the pairing $\bra\cdot\cdot\colon X^* \times X\dbs \to \RR$\big).   If $g \in \PCLSC(X^*)$ then, according to \eqref{FSTAR}, we define the Fenchel conjugate, $g^*$, of $g$ to be the function on $X\dbs$ given by
\begin{equation*}
g^*(x\dbs) := \supn_{X^*}\big[{x\dbs} - g\big] = \supn_{x^* \in X^*}\big[\bra{x^*}{x\dbs} - g(x^*)\big].
\end{equation*}
So, if $f \in \PCLSC(X)$ and we interpret $f\dbs$ to mean $(f^*)^*$ then $f\dbs$ is the function on $X\dbs$ given by
\begin{equation}
f\dbs(x\dbs) := \supn_{x^* \in X^*}\big[\bra{x^*}{x\dbs} - f^*(x^*)\big].
\end{equation}
If $x \in X$, we write $\wh x$ for the canonical image of $x$ in $X\dbs$, that is to say
\begin{equation*}
(x,x^*) \in X \times X^* \lr \bra{x^*}{\wh x} = \bra{x}{x^*}
\end{equation*}
If $g\colon\ X \to \rbar$, we write $\epi\,g$ for the {\em epigraph} of $g$,
\begin{equation*}
\{(x,\lambda) \in X \times \RR\colon\ g(x) \le \lambda\}.
\end{equation*}
\par
If $h \in \PC(X)$, the {\em lower semicontinuous envelope} of $h$, $\overh$, is defined by $\epi\,\overh = \overline{\epi\,h}$.   See \cite[p.\ 62]{ZBOOK}.   Of course, to make this definition legitimate, some effort has to be made to show that $\overline{\epi\,h}$ is the epigraph of a function.   Since $\epi\,\overh$ is closed, $\overh$ is lower semicontinuous.   It is worth pointing out that if $h$ is a discontinuous linear functional then $\overh = -\infty$ on $X$.
%:   \Thm{HKFthm}
\begin{theorem}\label{HKFthm}
Let $h \in \PC(X)$.   Let $k\colon\ X \to \rbar$ be lower semicontinuous and $k \le h$ on $X$. Then $k \le \overh \le h$ on $X$ and $\overh^* \le h^*$ on $X^*$.   It follows from this that $\overh \in \PCLSC(X)$ and $\overh^* = h^*$ on $X^*$. 
\end{theorem}
%:
\begin{proof}
We know from \cite[Theorem 2.2.6(i), p.\ 62]{ZBOOK} that $\overh$ is convex.   It follows from the hypotheses that $\epi\,h \subset \epi\,k$ and $\epi\,k$ is closed in $X \times \RR$. Consequently, $\epi\,h \subset \epi\,\overh = \overline{\epi\,h} \subset \epi\,k$, from which $k \le \overh \le h$ on $X$, as required.
\par
If $x^* \in X^*$ and $h^*(x^*) \in \RR$ then the Fenchel--Young inequality implies that $x^* - h^*(x^*) \le h$ on $X$, so $\epi\,h \subset \epi\,\big(x^* - h^*(x^*)\big)$.   Since $x^* - h^*(x^*)$ is continuous, $\epi\,\big(x^* - h^*(x^*)\big)$ is closed, thus $\epi\,\overh = \overline{\epi\,h} \subset \epi(x^* - h^*(x^*))$, from which $\overh \ge x^* - h^*(x^*)$ on $X$.   It follows easily that $\overh^*(x^*) \le h^*(x^*)$.   Of course, this inequality persists even if $h^*(x^*) = \infty$, and so we have proved that $\overh^* \le h^*$ on $X^*$. This completes the proof of \Thm{HKFthm}.  
\end{proof}
The main tool in the proof of \Thm{HKFthm} was epigraphical analysis.   The drawback of this method is that the definition of $\overh$ is not very intuitive.   We now discuss a more explicit geometric method of obtaining the function required for \Thm{THREEthm}, which we can actually express as a biconjugate. The preliminary work is done in \Lem{RSlem} below, which is of independent interest.  
\par
We shall use Rockafellar's version of the Fenchel duality theorem \big(which originally appeared in Rockafellar, \cite[Theorem~3(a), p.\ 85]{FENCHEL}\big) in the following form:  {\em Let $p,u \in \PC(X)$ and $u$ be continuous.   Then}
%:   \eqref{RTR1}
\begin{equation}\label{RTR1}
(p + u)^*(0) = \minn_{x^* \in X^*}\big[p^*(x^*) + u^*(-x^*)\big].
\end{equation}
We could have used instead K\"onig's sandwich theorem, a simple application of the Hahn--Banach theorem, see \cite[Theorem 1.7, p.\ 112]{KONIG}.
%:   \Lem{RSlem}
\begin{lemma}\label{RSlem}
Let $p \in \PC(X)$.  Let $s\colon\ X \to \rbar$ be lower semicontinuous, $s \le p$ on $X$ and $s(0) > 0$.   Then:
\par\noindent
{\rm (a)}\enspace  There exists $K \in [0,\infty[$ such that $p + K\|\cdot\| \ge 0$ on $X$.
\par\noindent
{\rm (b)}\enspace There exists $x^* \in X^*$ such that $p^*(x^*) \le 0$.
\end{lemma}
\begin{proof}
(a)\enspace Since the result is obvious if $p \ge 0$ on $X$, we can and will suppose that there exists
$w \in E$ such that $p(w) < 0$. Let $\theta \in \RR, \theta < s(w)$.   It follows that $\theta < s(w) \le p(w) < 0$.   Since $s$ is lower semicontinuous, there exists $m \ge 1$ such that $\inf_{y \in X,\ s(y) \le 0}\|y\| \ge \mth$\quand $\inf_{z \in X,\ s(z) \le \theta}\|z - w\| \ge \mth$.   Let $\alpha := p(w) - \theta > 0$.   Let $y \in X$.   We will show that
%:   \eqref{RS1}
\begin{equation}\label{RS1}
p(y) + \alpha m^2\|w\|\|y\| - \theta m\|y\| \ge 0.
\end{equation}
This gives the desired result, with $K := \alpha m^2\|w\| - \theta m$.
\par\noindent
{\bf Case 1.} ($p(y) \ge 0$)\enspace  In this case, \eqref{RS1} is obvious since $\alpha > 0$ and $\theta < 0$.
\par\noindent
{\bf Case 2.} ($\theta \le p(y) < 0$)\enspace  In this case, $s(y) < 0$, and so $\|y\| \ge \mth$, hence $m\|y\| - 1 \ge 0$.   Again since $\alpha > 0$ and $\theta < 0$,
\begin{equation*}
p(y) + \alpha m^2\|w\|\|y\| - \theta m\|y\| \ge \theta - \theta m\|y\| = (-\theta)(m\|y\| - 1) \ge 0,
\end{equation*}
which gives \eqref{RS1}.
\par\noindent
{\bf Case 3.} ($p(y) < \theta$)\enspace    Let $\beta := \theta - p(y) > 0$.   $\beta$ (unlike $\alpha$) depends on $y$.   Here, the convexity of $p$ and the fact that $s \le p$ on $X$ imply that
\begin{equation*}
s\left(\frac{\alpha y + \beta w}{\alpha + \beta}\right) \le p\left(\frac{\alpha y + \beta w}{\alpha + \beta}\right) \le \frac{\alpha p(y) + \beta p(w)}{\alpha + \beta} = \frac{\alpha(\theta - \beta) + \beta(\alpha + \theta)}{\alpha + \beta} = \theta.
\end{equation*}
Thus, from the choice of $m$ again,   
\begin{equation*}
\frac{\alpha(\|y\| +  \|w\|)}{\alpha + \beta} \ge \left\|\frac{\alpha(y - w)}{\alpha + \beta}\right\| = \left\|\frac{\alpha y + \beta w}{\alpha + \beta} - w\right\| \ge \frac1m.
\end{equation*}
This is equivalent to the statement $\alpha m\|w\| + \alpha m\|y\| - \alpha - \beta \ge 0$.   Substituting $\beta = \theta - p(y)$, we see that
%:   \eqref{RS2}
\begin{equation}\label{RS2}
p(y) \ge \theta + \alpha - \alpha m\|w\| - \alpha m\|y\|.
\end{equation}
We still have $m\|y\| - 1 \ge 0$, and also $\alpha m\|w\| - \theta - \alpha =  \alpha m\|w\| - p(w) > 0$.  It follows that $(\alpha m\|w\| - \theta - \alpha)(m\|y\| - 1) \ge 0$.   Equivalently, 
\begin{equation*}
\alpha m^2\|w\|\|y\| - \theta m\|y\| \ge \alpha m\|y\| +  \alpha m\|w\| - \theta - \alpha.
\end{equation*}
\eqref{RS1} now follows by adding this to \eqref{RS2}.
\par
Now let $u := K\|\cdot\|$.   From (a), $p + u \ge 0$ on $X$, and so $(p + u)^*(0) \le 0$.   \eqref{RTR1} now gives $x^* \in X^*$ such that $p^*(x^*) + u^*(-x^*) \le 0$.   Since $u(0) = 0$, $u^*(-x^*) \ge 0$, and thus we obtain (b). 
\end{proof}
%:   \Thm{HKF2thm}
\begin{theorem}\label{HKF2thm}
Let $h \in \PC(X)$.   Let $k\colon\ X \to \rbar$, be lower semicontinuous and $k \le h$ on $X$. For all $x \in X$, let $f(x) := \supn_{x^* \in X^*}\big[\bra{x}{x^*} - h^*(x^*)\big]$, {\em i.e.}, $f(x) := h\dbs(\wh x)$.  Then:
\par\noindent
{\rm (a)}\enspace  $f \ge k$ on $X$, and so $f\colon\ X \to \rbar$.
\par\noindent
{\rm (b)}\enspace $f \in \PCLSC(X)$ and $f^* = h^*$ on $X^*$.
\end{theorem}
\begin{proof}
(a)\enspace Let $x \in X$, $\lambda \in \RR$ and $\lambda < k(x)$.   Let $p(y) := h(y + x) - \lambda$ and $s(y) := k(y + x) - \lambda$, so $s(0) = k(x) - \lambda > 0$.   \Lem{RSlem}(b) now gives $x^* \in X^*$ such that $p^*(x^*) \le 0$.   It is easily seen that this is equivalent to the statement that $\bra{x}{x^*} - h^*(x^*) \ge \lambda$.   (a) now follows by letting $\lambda \to k(x)$.     
\par
(b)\enspace From the Fenchel--Young inequality, for all $x^* \in X^*$, ${x^*} - h^*(x^*) \le h$ on $X$, thus $f \le h$ on $X$, and so $f \in \PCLSC(X)$ and $f^* \ge h^*$ on $X^*$.   On the other hand, for all $x^* \in X^*$, $f \ge {x^*} - h^*(x^*)$  on $X$, {\em i.e.}, $ {x^*} - f \le h^*(x^*)$ on $X$ thus, for all $x^* \in X^*$, $f^*(x^*) = \supn_{X}\big[{x^*} - f\big] \le h^*(x^*)$ on $X$.   Thus $f^* = h^*$ on $X^*$, completing the proof of (b).
\end{proof}
%:\Sec{EEsec}
\section{$E \times E^*$, $q_L$, $r_L$ and quasidensity}\label{EEsec}
Now let $E$ be nonzero Banach space.   For all $(x,x^*) \in E \times E^*$, let\quad $\|(x,x^*)\| := \sqrt{\|x\|^2 + \|x^*\|^2}$,\quad and represent $(E \times E^*)^*$ by $E^* \times E\dbs$, under the pairing 
\begin{equation*}
\Bra{(x,x^*)}{(y^*,y\dbs)} := \bra{x}{y^*} + \bra{x^*}{y\dbs}.
\end{equation*}
Define the linear map $L\colon\ E \times E^* \to E^* \times E\dbs$ by
%:   \eqref{Ldef}
\begin{equation}\label{Ldef}
L(x,x^*) := (x^*,\wh{x}).
\end{equation}
Then
\begin{equation*}
\all\ a,b \in E \times E^*,\quad \bra{a}{Lb} = \bra{b}{La}.
\end{equation*}
%:   \Lem{EXNlem}
We define the even real functions $q_L$ and $r_L$ on $E \times E^*$ by\quad $q_L(x,x^*) := \bra{x}{x^*}$\quad and
%:   \eqref{RL1}
\begin{equation}\label{RL1}
r_L(x,x^*) := \half\|x\|^2 + \half\|x^*\|^2 + \bra{x}{x^*} = \half\|(x,x^*)\|^2 + q_L(x,x^*).
\end{equation}
For all $(x,x^*) \in E \times E^*$,\quad $|q_L(x,x^*)| = |\bra{x}{x^*}| \le \|x\|\|x^*\| \le \half\|(x,x^*)\|^2$,\quad so
%:   \eqref{RL3}
\begin{equation}\label{RL3}
0 \le r_L \le \|\cdot\|^2\ \on\ E \times E^*.
\end{equation}
We note for future reference that,
%:   \eqref{RL2}
\begin{equation}\label{RL2}
\all\ b,c \in E \times E^*,\quad q_L(b - c) = q_L(b) + q_L(c) - \bra{b}{Lc}.
\end{equation}
%:   \Def{QDdef}
\begin{definition}\label{QDdef}
Let $A \subset E \times E^*$.   We say that $A$ is {\em quasidense} (in $E \times E^*$) if
%:   \eqref{QD1}
\begin{equation}\label{QD1}
c \in E \times E^* \qlr \inf r_L(A - c) \le 0 \iff \inf r_L(A - c) = 0.
\end{equation}
(The ``$\iff$'' above follows since $r_L \ge 0$.)  In longhand, \eqref{QD1} can be rewritten:
\begin{equation}\label{EE1}
(x,x^*) \in E \times E^*\lr
\inf_{(s,s^*) \in A}\big[\half\|s - x\|^2 + \half\|s^* - x^*\|^2 + \bra{s - x}{s^* - x^*}\big] \le 0.
\end{equation}
\end{definition}
%:  \Ex{SUBex}
\begin{example}[Subdifferentials]\label{SUBex}
Let $f\colon\ E \to \rbar$ be proper, convex and lower semicontinuous and $\partial f$ be the usual subdifferential.   Then $G(\partial f)$ is\break quasidense.   There is an ``elementary'' proof of this in \cite[Theorem 4.6]{PARTTWO}.   There is also a more sophisticated proof based on \Thm{FSTARthm} below in \cite[Theorem 7.5, p.\ 1033]{PARTONE}.   We shall see in \Thm{RLMAXthm} below that this result generalizes Rockafellar's maximal monotonicity theorem. 
\par
In fact, the ``elementary'' proof mentioned above can be generalized to some more general subdifferentials for non--convex functions.   See Simons--Wang,\break \cite[Definition 2.1, p.\ 633]{SW} and \cite[Theorem 3.2, pp.\ 634--635]{SW}. 
\end{example}
The dual norm on $E^* \times E\dbs$ is given by  $\|(y^*,y\dbs)\| := \sqrt{\|y^*\|^2 + \|y\dbs\|^2}$.  Define the linear map $\LT\colon\ E^* \times E\dbs \to E\dbs \times E\trs$ by $\LT(x^*,x\dbs) := \big(x\dbs,\wh{x^*}\big)$.   Then $\qLt(y^*,y\dbs) = \bra{y^*}{y\dbs}$ and $\rLt(y^*,y\dbs) := \half\|y^*\|^2 + \half\|y\dbs\|^2 + \bra{y^*}{y\dbs} = \half\|(y^*,y\dbs)\|^2 + \qLt(y^*,y\dbs)$.
\par
One can easily verify the following generalization of \eqref{RL2}:
%:   \eqref{QD2}
\begin{equation}\label{QD2}
c \in E \times E^*\hbox{ and }c^* \in E^* \times E\dbs \lr\qLt(c^* + Lc) = \qLt(c^*) + \bra{c}{c^*} + q_L(c).
\end{equation}
\smallbreak
\Lem{SLlem} below gives a very nice relationship between $L$ and quasidensity.   It is the first of two preliminary results leading to the main result of this section, \Thm{NIthm}.
%:   \Lem{SLlem}
\begin{lemma}\label{SLlem}
$L(E \times E^*)$ is quasidense in $E^* \times E\dbs$.   In other words:
%:   \eqref{SL3}
\begin{equation}\label{SL3}
c^* \in E^* \times E\dbs\qlr\infn_{c \in E \times E^*}\rLt(Lc - c^*) = 0.
\end{equation}
{\rm In longhand, this can be rewritten:} for all $(y^*,y\dbs) \in E^* \times E\dbs$,
%:   \eqref{SL2}
\begin{equation}\label{SL2}
\infn_{(x,x^*) \in E \times E^*}\big[\half\|y^* - x^*\|^2 + \half\|y\dbs - \wh x\|^2 + \bra{y^* - x^*}{y\dbs - \wh x}\big] = 0.
\end{equation}
\end{lemma}
\begin{proof}
Let $(y^*,y\dbs) \in E^* \times E\dbs$.   For all $\eps > 0$, the definition of $\|y\dbs\|$ provides $z^* \in E^*$ such that $\|z^*\| \le \|y\dbs\|$ and $\bra{z^*}{y\dbs} \le -\|y\dbs\|^2 + \eps$, from which\quad $\half\|z^*\|^2 + \half\|y\dbs\|^2 + \bra{z^*}{y\dbs} \le \|y\dbs\|^2 + \bra{z^*}{y\dbs} \le \eps$.\quad So
\begin{align*}
0 &\le \infn_{(x,x^*) \in E \times E^*}\big[\half\|y^* - x^*\|^2 + \half\|y\dbs - \wh x\|^2+ \bra{y^* - x^*}{y\dbs - \wh x}\big]\\ 
&= \infn_{(x,z^*) \in E \times E^*}\big[\half\|z^*\|^2 + \half\|y\dbs - \wh x\|^2 + \bra{z^*}{y\dbs - \wh x}\big]\\
&\le \infn_{z^* \in E^*}\big[\half\|z^*\|^2 + \half\|y\dbs\|^2 + \bra{z^*}{y\dbs}\big] \le 0.
\end{align*}
This establishes \eqref{SL2}, and hence \eqref{SL3}.
\end{proof}
%
%:   \Lem{BBlem}
\begin{lemma}\label{BBlem}
Let $b \in E \times E^*$ and $b^* \in E^* \times E\dbs$.   Then
%:   \eqref{BB1}
\begin{equation}\label{BB1}
\qLt(Lb + b^*) \le r_L(b) + r_\LT(b^*).
\end{equation}
Let $a,c \in E \times E^*$ and $c^* \in E^* \times E\dbs$.   Then
%:   \eqref{BB2}
\begin{equation}\label{BB2}
\qLt(La - c^*) \le r_L(a - c) + r_\LT(Lc - c^*).
\end{equation}
\end{lemma}
\begin{proof}
From \eqref{QD2},
\begin{gather*}
r_L(b) + \rLt(b^*) - \qLt(Lb + b^*)\\
= q_L(b) + \half\|b\|^2 + \qLt(b^*) + \half\|b^*\|^2 - q_L(b) - \bra{b}{b^*} - \qLt(b^*)\\
= \half\|b\|^2 + \half\|b^*\|^2 - \bra{b}{b^*} \ge \half\|b\|^2 + \half\|b^*\|^2 - \|b\|\|b^*\| \ge 0.
\end{gather*}
This completes the proof of \eqref{BB1}, and \eqref{BB2} follows from \eqref{BB1} with $b := a - c$ and $b^* := Lc - c^*$  
\end{proof}
We have the following fundamental result:
%:   \Thm{NIthm}
\begin{theorem}\label{NIthm}
Let $A \subset E \times E^*$ and $A$ be quasidense in $E \times E^*$.   Then, for all $c^* \in E^* \times E\dbs$, $\inf\qLt(L(A) - c^*) \le 0$.
\end{theorem}
\begin{proof}
Let $c^* \in E^* \times E\dbs$ and $\eps > 0$.   Then, from \Lem{SLlem} and \Def{QDdef}, there exist $c \in E \times E^*$ and then $a \in A$ such that $\rLt(Lc - c^*) < \half\eps$ and $r_L(a - c) < \half\eps$.  From \eqref{BB2}, $\qLt(La - c^*) < \eps$.
\end{proof}
The following definition was made in \cite[Definition 10,\ p.\ 183]{RANGE}:
\begin{definition}
Let $A \subset E \times E^*$.   Then $A$ is of {\em type (NI)} if,
%:   \eqref{NI1}
\begin{equation}\label{NI1}
\all\ (y^*,y\dbs) \in E^* \times E\dbs,\quad\infn_{(s,s^*) \in A}\bra{s^* - y^*}{\wh s - y\dbs} \le 0.
\end{equation}
In our current notation, \eqref{NI1} can be rephrased as
%:   \eqref{NI2}
\begin{equation}\label{NI2}
\all\ c^* \in E^* \times E\dbs,\quad\infn_{a \in A}\qLt(La - c^*) \le 0.
\end{equation}
``(NI)'' stands for ``negative infimum''.   We note that $A$ is not constrained to be monotone in this definition. 
\end{definition}
%:   \Cor{NIcor}
\begin{corollary}\label{NIcor}
Let $A \subset E \times E^*$ and $A$ be quasidense in $E \times E^*$.   Then $A$ is of type (NI).
\end{corollary}
\begin{proof}
This is immediate from \Thm{NIthm} and \eqref{NI2}.
\end{proof}
There is another way of viewing \Thm{NIthm}.   In order to explain this, we introduce the function $\Theta_A$.   (Compare \cite[Definition 6.2, p.\ 1029]{PARTONE}.)
%:   \Def{THAdef} [The definition of $\theta_S$]
\begin{definition}\label{THAdef}
Let $A \subset E \times E^*$ and $A \ne \emptyset$.   We define the function\break $\Theta_A\colon\ E^* \times E\dbs \to \rbar$ by:
\begin{equation*}
\all\ c^* \in E^* \times E\dbs,\quad \Theta_A(c^*) := \supn_{A}\big[c^* - q_L\big] = \supn_{a \in A}\big[\bra{a}{c^*} - q_L(a)\big].
\end{equation*}
In longhand: for all $(y^*,y\dbs) \in E^* \times E\dbs$,
%:   \eqref{THALONG}
\begin{equation*}
\Theta_A(y^*,y\dbs) := \supn_{(s,s^*) \in A}\big[\bra{s}{y^*} + \bra{s^*}{y\dbs} - \bra{s}{s^*}\big].
\end{equation*}
\end{definition}
%:   \Cor{THAcor}
\begin{corollary}\label{THAcor}
Let $A \subset E \times E^*$ and $A$ be quasidense in $E \times E^*$.   Then $\Theta_A \ge \qLt$ on $E^* \times E\dbs$.
\end{corollary}
\begin{proof}
Let $c^* \in E^* \times E\dbs$.   Then, from \Def{THAdef} and \eqref{QD2},
\begin{align*}
\Theta_A(c^*) - \qLt(c^*)
&= \supn_{a \in A}\big[\bra{a}{c^*} - q_L(a) - \qLt(c^*)\big]\\
&= -\infn_{a \in A}\big[\qLt(c^*) - \bra{a}{c^*} + q_L(a)\big] = -\infn_{a \in A}\qLt(La - c^*).
\end{align*}
The result now follows since, from \Thm{NIthm}, $\infn_{a \in A}\qLt(La - c^*) \le 0$. 
\end{proof}
%:   \Rem{NIrem}
\begin{remark}\label{NIrem}
\Cor{THAcor} will be used in \Lem{FSTARlem} and \Thm{THthm}.   The converses of \Cors{NIcor} and \ref{THAcor} are true for maximally monotone sets. (See \Thm{THthm}).   We give an example where the converse of \Cor{NIcor} fails without the hypothesis of maximal monotonicity in \Ex{QDneNI} below. \Ex{QDneNI} depends on the following simple fact:  
\end{remark}
\begin{fact}\label{NIFACT}
{\em Let $E$ be reflexive, $S\colon\ E \toto E^*$, $D(S) = E$ and $A := G(S)$.   Then $A$ is of type (NI).}
\end{fact}
\begin{proof}
Let $(y^*,y\dbs) \in E^* \times E\dbs$.   Since $E$ is reflexive, there exists $s \in E$ such that $\wh s = y\dbs$.   Since $D(S) = E$ and $G(S) = A$, there exists $s^* \in E^*$ such that $(s,s^*) \in A$.   Since $\bra{s^* - y^*}{\wh s - y\dbs} = \bra{s^* - y^*}{0} = 0$, $A$ is of type (NI).
\end{proof}
%: \Ex{QDneNI}
\begin{example}\label{QDneNI}
Let $E = \RR$.   If $(s,s^*),(x,x^*) \in \RR \times \RR$ then   
\begin{align*}
r_L\big((s,s^*) - (x,x^*)\big) &= \half\|s - x\|^2 + \half\|s^* - x^*\|^2 + \bra{s - x}{s^* - x^*}\\
&= \half(s - x)^2 + \half(s^* - x^*)^2 + (s - x)(s^* - x^*)\\
&= \half(s + s^* - x - x^*)^2
\end{align*}
Let $A := \big\{(\lambda,-\lambda)\colon \lambda \in \RR\big\} \subset \RR \times \RR$ and $(x,x^*) := (1,0) \in \RR \times \RR$ then, for all $(s,s^*) \in A$, $r_L\big((s,s^*) - (1,0)\big) = \half(s - s - 1 - 0)^2 = \half$.   Thus $A$ is not quasidense.   However, from Fact \ref{NIFACT}, $A$ is of type (NI).
\end{example}
%:\Sec{RLsec}
\section{Quasidense sets determined by the coincidence sets of convex functions}\label{RLsec}
%:   \Def{FCdef}
\begin{definition}\label{FCdef}
If $f \in \PC(E \times E^*)$ and $f \ge q_L$ on $E \times E^*$, we write $\coinc[f]$ for the ``coincidence set''
\begin{equation*}
\big\{b \in E \times E^*\colon\ f(b) = q_L(b)\big\}.
\end{equation*}
The notation ``$M_f$'' has been used for this set in the literature.   We have avoided the ``$M_f$'' notation because it lead to superscripts and subscripts on subscripts, and consequently makes the analysis harder to read.  If $g$ is a proper, convex function on $E^* \times E\dbs$ and  $g \ge \qLt$ on $E^* \times E\dbs$, we write $\dcoinc[g]$ for the ``dual coincidence set''
\begin{equation*}
\big\{b^* \in E^* \times E\dbs\colon\ g(b^*) = \qLt(b^*)\big\}.
\end{equation*}
\end{definition}
\Lems{EXNlem} and \ref{RLlem} lead to the main result of the section, \Thm{FCthm}:
%:   \Lem{EXNlem}
\begin{lemma}[A boundedness result]\label{EXNlem}
Let $X$ be a nonzero real Banach space and $g \in \PC(X)$.   Suppose, further, that $\infn_{x \in X}\big[g(x) + \half\|x\|^2\big] = 0$, $y,z \in X$, $g(y) + \half\|y\|^2 \le 1$ and $g(z) + \half\|z\|^2 \le 1$.   Then $\|y\| \le \|z\| + \sqrt8$.
\end{lemma}
\begin{proof}
We have $\eighth\big[\|y\| - \|z\|\big]^2 = \fourth\|y\|^2 + \fourth\|z\|^2 - \eighth\big[\|y\| + \|z\|\big]^2$ and
\begin{equation*}
0 \le g\big(\half y + \half z\big) + \half\|\half y + \half z\|^2 \le \half g(y) + \half g(z) + \eighth\big[\|y\| + \|z\|\big]^2.
\end{equation*}
Thus, by addition,
\begin{align*}
\eighth\big[\|y\| - \|z\|\big]^2 \le \half g(y) + \half g(z) + \fourth\|y\|^2 + \fourth\|z\|^2 \le \half + \half = 1.
\end{align*}
This gives the required result.
\end{proof}
%:  \Lem{RLlem}
\begin{lemma}\label{RLlem}
Let $b,d \in E \times E^*$.   Then: 
\begin{equation*}
r_L(b + d) \le r_L(b) + 2\|b\|\|d\| + r_L(d)
\le \|b\|^2 + 2\|b\|\|d\| + r_L(d).
\end{equation*}
\end{lemma}
\begin{proof}
Let $b = (x,x^*)$ and $d = (z,z^*)$.   From the Cauchy--Schwarz inequality, we have\quad $\|x\|\|z\| + \|x^*\|\|z^*\| \le \sqrt{\|x\|^2 + \|x^*\|^2}\sqrt{\|z\|^2 + \|z^*\|^2} = \|b\|\|d\|$.\quad  From the triangle inequality,\quad $\|x + z\|^2 \le  \big(\|x\| + \|z\|\big)^2 = \|x\|^2 + 2\|x\|\|z\| + \|z\|^2$\quad  and\quad $\|x^* + z^*\|^2 \le \big(\|x^*\| + \|z^*\|\big)^2 = \|x^*\|^2 + 2\|x^*\|\|z^*\| + \|z^*\|^2$.   Thus
\begin{align*}
\half\|b + d\|^2 &= \half\|x + z\|^2 + \half\|x^* + z^*\|^2\\
&\le \half\|x\|^2 + \|x\|\|z\| + \half\|z\|^2 + \half\|x^*\|^2 + \|x^*\|\|z^*\| + \half\|z^*\|^2\\
&\le \half\|b\|^2 + \|b\|\|d\| + \half\|d\|^2.
\end{align*}
Also, from \eqref{RL2} with $c := -d$ and the fact that $\|Ld\| = \|d\|$,
\begin{align*}
q_L(b + d) =  q_L(b) + \bra{b}{Ld} + q_L(d) \le q_L(b) + \|b\|\|d\| + q_L(d).
\end{align*}
The result now follows by addition, \eqref{RL1} and \eqref{RL3}.
\end{proof}
%:   \Thm{FCthm} [Primal conditions for quasidensity]
\begin{theorem}[Primal condition for quasidensity]\label{FCthm}
Let $f \in \PCLSC(E \times E^*)$ and $f \ge q_L$ on $E \times E^*$.   For all $c,b \in E \times E^*$, let
%:   \eqref{FC0}
\begin{equation}\label{FC0}
f_c(b) := f(b + c) - \bra{b}{Lc} - q_L(c) =  (f - q_L)(b + c) + q_L(b) \ge q_L(b).
\end{equation}
{\em \big(The first expression shows that $f_c \in \PCLSC(E \times  E^*$).\big)} Then {\rm(a)}$\iff${\rm(b)}:
\par\noindent
{\rm(a)}\enspace $\coinc[f]$ is quasidense.
\par\noindent
{\rm(b)}\enspace For all $c \in E \times E^*$, $\infn_{b \in E \times E^*}\big[f_c(b) + \half\|b\|^2\big] \le 0$.
\end{theorem}
\begin{proof}
Let $A := \coinc[f]$.   Let $c \in E \times E^*$.   Since $f_c = q_L$ on $A - c$,
\begin{align*}
\infn_{b \in E \times E^*}\big[f_c(b) + \half\|b\|^2\big]
&\le \infn_{b \in A - c}\big[f_c(b) + \half\|b\|^2\big]\\
&= \infn_{b \in A - c}\big[q_L(b) + \half\|b\|^2\big]\\
&= \infn_{b \in A - c}r_L(b) = \inf r_L(A - c) = 0,
\end{align*}
and so it follows that (a)$\lr$(b).
\par
Suppose now that (b) is satisfied and $c \in E \times E^*$.   Let $c_0 := c$, so that $\infn_{b \in E \times E^*}\big[f_{c_0}(b) + \half\|b\|^2\big] \le 0$.  From \eqref{FC0}, $f_{c_0} + \half\|\cdot\|^2 \ge q_L + \half\|\cdot\|^2 \ge 0$ on $E \times E^*$, so in fact $\infn_{b \in E \times E^*}\big[f_{c_0}(b) + \half\|b\|^2\big] = 0$.   From \Lem{EXNlem}, there exists $M \in \RR$ such that
%:   \eqref{FC1}
\begin{equation}\label{FC1}
f_{c_0}(b) + \half\|b\|^2 \le 1 \qlr \|b\| \le M.
\end{equation}
Let $0 \le \eps < 1$.  Let $1 \ge \eps_1 \ge \eps_3 \ge \eps_3 \dots > 0$ and $\sum_{n = 1}^\infty \eps_n \le \eps$.   We now define inductively $c_1,c_2,\dots \in E \times E^*$.   Suppose that $n \ge 0$ and $c_{n}$ is known.   By hypohesis, $\infn_{b \in E \times E^*}\big[f_{c_{n}}(b) + \half\|b\|^2\big] \le 0$, and so there exists $b_n \in E \times E^*$ such that $f_{c_{n}}(b_n) + \half\|b_n\|^2  \le \eps_{n + 1}^2$.   Let $c_{n + 1} := b_n + c_{n}$.   This completes the inductive construction.
\par
Since $b_n = c_{n + 1} - c_{n}$, we now have $c_0,c_1,c_2,\dots,$ such that,
%:   \eqref{FC2}
\begin{equation}\label{FC2}
\all\ n \ge 0,\quad f_{c_{n}}(c_{n + 1} - c_{n}) + \half\|c_{n + 1} - c_{n}\|^2  \le \eps_{n + 1}^2.
\end{equation}
From \eqref{FC0}, $f_{c_{n}}(c_{n + 1} - c_{n}) = (f - q_L)(c_{n + 1}) + q_L(c_{n + 1} - c_{n})$ and so,
%:   \eqref{FC3}
\begin{equation*}
\all\ n \ge 0,\quad (f - q_L)(c_{n + 1}) + r_L(c_{n + 1} - c_{n})  \le \eps_{n + 1}^2.
\end{equation*}
Since $f \ge q_L$ and, from \eqref{RL3}, $r_L \ge 0$ on $E \times E^*$, this implies that,
%:   \eqref{FC4}
\begin{equation}\label{FC4}
\all\ n \ge 0,\quad (f - q_L)(c_{n + 1}) \le \eps_{n + 1}^2 \quand r_L(c_{n + 1} - c_{n}) \le \eps_{n + 1}^2.
\end{equation}
We now prove that,
%:   \eqref{FC5}
\begin{equation}\label{FC5}
\all\ n \ge 1,\quad\|c_{n + 1} - c_{n}\| \le \sqrt{10}\eps_{n}.
\end{equation}
Let $n \ge 1$.  Since $f$ is convex,  \eqref{FC4} gives
\begin{equation*}
2f(\half c_{n + 1} + \half c_{n}) \le f(c_{n + 1}) + f(c_{n}) \le q_L(c_{n + 1}) + \eps_{n + 1}^2 + q_L(c_{n}) + \eps_{n}^2. 
\end{equation*}
Since $f \ge q_L$ on $E \times E^*$ and $\eps_{n + 1}^2 \le \eps_{n}^2$, it follows that
\begin{equation*}
2q_L(\half c_{n + 1} + \half c_{n}) - q_L(c_{n + 1}) - q_L(c_{n}) \le 2\eps_{n}^2. 
\end{equation*}
Thus, from the quadraticity of $q_L$, $\half q_L(c_{n + 1} + c_{n}) - q_L(c_{n + 1}) - q_L(c_{n}) \le 2\eps_{n}^2$.   Since $q_L(c_{n + 1}) + q_L(c_{n}) = \half q_L(c_{n + 1} + c_{n}) + \half q_L(c_{n + 1} - c_{n})$, we see that
\begin{equation*}
- q_L(c_{n + 1} - c_{n}) \le 4\eps_{n}^2.
\end{equation*}
From \eqref{FC4}, $q_L(c_{n + 1} - c_{n}) + \half\|c_{n + 1} - c_{n}\|^2 = r_L(c_{n + 1} - c_{n}) \le \eps_{n + 1}^2$.   Thus
\begin{equation*}
\all\ n \ge 1,\quad\half\|c_{n + 1} - c_{n}\|^2 \le 4\eps_{n}^2 + \eps_{n + 1}^2 \le 5\eps_{n}^2.
\end{equation*}
Thus we obtain \eqref{FC5}.   We will also need an estimate for $\|c_{1} - c_{0}\|$. This is not covered by \eqref{FC5}.   Now \eqref{FC5} used the inequality $f(c_n) \le q_L(c_{n}) + \eps_{n}^2$.  A similar analysis for $\|c_{1} - c_{0}\|$ is unlikely, because we have no knowledge about $f(c_0)$ --- there is no {\em a priori} reason why $f(c_0)$ should even be finite.   This issue is partially resolved by \eqref{FC7} below.
\par
It follows from \eqref{FC5} that $\lim_{n \to \infty}c_n$ exists.  Let $a_\eps := \lim_{n \to \infty}c_n$.   Clearly, $a_\eps - c_1 = \sum_{n = 1}^\infty (c_{n + 1} - c_{n})$ and so, from \eqref{FC5},
%:   \eqref{FC6}
\begin{equation}\label{FC6}
\left.\begin{aligned}
\|a_\eps - c_{1}\| &= \big\|\ts\sum_{n = 1}^\infty (c_{n + 1} - c_{n})\big\|\\
&\le\ts \sum_{n = 1}^\infty \|c_{n + 1} - c_{n}\| \le \sqrt{10}\sum_{n = 1}^\infty\eps_{n} \le 4\eps.
\end{aligned}
\right\}
\end{equation}
From \eqref{FC4}, the lower semicontinuity of $f$, and the continuity of $q_L$, $f(a_\eps) \le q_L(a_\eps)$, and so $a_\eps \in \coinc[f]$.   We must now estimate $r_L(a_\eps - c)$.  \eqref{FC2} with $n = 0$ gives $f_{c_0}(c_{1} - c_{0}) + \half\|c_{1} - c_{0}\|^2  \le \eps_{1}^2 \le 1$ and so, from \eqref{FC1}, 
%:   \eqref{FC7}
\begin{equation}\label{FC7}
\|c_{1} - c\| = \|c_{1} - c_{0}\| \le M.
\end{equation}
Furthermore, \eqref{FC4} with $n = 0$ gives
%:   \eqref{FC8}
\begin{equation}\label{FC8}
r_L(c_{1} - c) = r_L(c_{1} - c_{0}) \le \eps_{1}^2 \le \eps.
\end{equation}
From \Lem{RLlem} with $b = a_\eps - c_1$ and $d = c_1 - c$, \eqref{FC6}, \eqref{FC7} and \eqref{FC8},
%x
\begin{align*}
r_L(a_\eps - c) &\le \|a_\eps - c_1\|^2 + 2\|a_\eps - c_1\|\|c_1 - c\| + r_L(c_1 - c)\\
&\le 16\eps^2 + 8\eps M + \eps \le 16\eps + 8\eps M + \eps = (17 + 8M)\eps.
\end{align*}
Letting $\eps \to 0$, we see that\quad $\inf r_L(\coinc[f] - c) \le 0$.\quad   Thus $\coinc[f]$ is\break quasidense, and (a) holds.
\end{proof}
\begin{remark}
An inspection of the above proof shows that we have, in fact, proved that if $\coinc[f]$ is quasidense then $\coinc[f]$ satisfies the stronger\break condition that, for all $c \in E \times E^*$, there exists $K_c \ge 0$ such that
\begin{equation*}
\inf\big\{r_L(a - c)\colon\ a \in \coinc[f],\ \|a - c\| \le K_c\big\} \le 0.
\end{equation*}
\end{remark}
It is clear from \eqref{FC0} that, for all $b,c \in E \times E^*$, $(f_c - q_L)(b) = (f - q_L)(b + c)$.   In light of this, the result of \Lem{FClem} below is very pleasing:
%:   \Lem{FClem}
\begin{lemma}\label{FClem}
Let $f \in \PCLSC(E \times E^*)$ and $f_c$ be as in \eqref{FC0}.   Then, for all $c \in E \times E^*$ and $b^* \in E^* \times E\dbs$, $({f_c}^* - \qLt)(b^*) = (f^* - \qLt)(b^* + Lc)$.
\end{lemma}
\begin{proof}
From \eqref{FC0}, the substitution $d = b + c$, \eqref{RL2} and \eqref{QD2},
\begin{align*}
({f_c}^* &- \qLt)(b^*)
= \supn_{b \in E \times E^*}\big[\bra{b}{b^*} - (f - q_L)(b + c) - q_L(b) - \qLt(b^*)\big]\\
&= \supn_{d \in E \times E^*}\big[\bra{d - c}{b^*} - (f - q_L)(d) - q_L(d - c)- \qLt(b^*)\big]  \\
&= \supn_{d \in E \times E^*}\big[\bra{d}{b^*} - f(d) + q_L(d) - q_L(d - c) - \bra{c}{b^*} - \qLt(b^*)\big]\\
&= \supn_{d \in E \times E^*}\big[\bra{d}{b^* + Lc} - f(d) - q_L(c) - \bra{c}{b^*} - \qLt(b^*)\big]\\
&= f^*(b^* + Lc) - \qLt(b^* + Lc).
\end{align*}
This gives the required result.
\end{proof}
%:   \Lem{FSTARlem}
\begin{lemma}\label{FSTARlem}
Let $f \in \PC(E \times E^*)$, $f \ge q_L$ on $E \times E^*$ and $\coinc[f]$ be quasidense.   Then $f^* \ge \qLt$ on $E^* \times E\dbs$.
\end{lemma}
\begin{proof}
Let $A := \coinc[f]$.   Let $c^* \in E^* \times E\dbs$.   Then, since $f = q_L$ on $A$,
\begin{equation*}
f^*(c^*) = \supn_{E \times E^*}\big[{c^*} - f\big] \ge \supn_{A}\big[{c^*} - f\big] = \supn_{A}\big[{c^*} - q_L\big].
\end{equation*}
Thus, from  \Def{THAdef} and \Cor{THAcor}, $f^*(c^*) \ge \Theta_A(c^*) \ge \qLt(c^*)$.
\end{proof}
%:   \Thm{FSTARthm}[Dual condition for quasidensity]
\begin{theorem}[Dual condition for quasidensity]\label{FSTARthm}
Let $f \in \PCLSC(E \times E^*)$ and $f \ge q_L$ on $E \times E^*$.   Then\quad $\coinc[f]$ is quasidense $\iff f^* \ge \qLt$ on $E^* \times E\dbs$.
\end{theorem}
\begin{proof}
By virtue of \Lem{FSTARlem}, we only have to prove the implication ($\rl$).   So assume that $f^* \ge \qLt$ on $E^* \times E\dbs$.  Let $c \in E \times E^*$.  Let $f_c \in \PCLSC(E \times E^*)$ be as in \eqref{FC0}.   From \Lem{FClem}, ${f_c}^* \ge \qLt \ge -\half\|\cdot\|^2$ on $E^* \times E\dbs$, thus ${f_c}^* + \half\|\cdot\|^2 \ge 0$ on $E^* \times E\dbs$.   We now derive from \eqref{RTR1} that $\infn_{b \in E \times E^*}\big[f_c(b) + \half\|b\|^2\big] \le 0$.   Thus, from \Thm{FCthm}, $\coinc[f]$ is quasidense, as required.
\end{proof}
%:   \Def{FATdef}
\begin{definition}\label{FATdef}
Let $f \in \PC(E \times E^*)$.   We define the function $f^@$ on $E \times E^*$ by $f^@ := f^* \circ L$.   Explicitly, for all $a \in E \times E^*$,
%:  \eqref{FAT}
\begin{equation}\label{FAT}
f^@(a) := \supn_{E \times E^*}\big[{La} - f\big] = \supn_{b \in E \times E^*}\big[\bra{b}{La} - f(b)\big].
\end{equation}
\end{definition}
\Lem{Llem} will be used in \Thm{THREEthm}, \Lem{PSlem} and \Thm{COINCthm}. 
%:   \Lem{Llem}
\begin{lemma}\label{Llem}
Let $f,f^@ \in \PC(E \times E^*)$, $f \ge q_L$ and $f^@ \ge q_L$ on $E \times E^*$.   Then $\coinc[f] \subset \coinc[f^@]$. 
\end{lemma}
\begin{proof}
Let $a \in \coinc[f]$, $b \in \dom\,f$, $\lambda,\mu > 0$ and $\lambda + \mu = 1$.    Then
\begin{align*}
\lambda\mu q_L(a)
&= \mu q_L(a) - \mu^2q_L(a)
= \mu f(a) - \mu^2q_L(a)\\
&\ge f(\lambda b + \mu a) - \lambda f(b) - \mu^2q_L(a)
\ge q_L(\lambda b + \mu a) - \lambda f(b) - \mu^2q_L(a)\\
&= \lambda^2q_L(b) + \lambda\mu\bra{b}{La} - \lambda f(b).
\end{align*}
Dividing by $\lambda$ and letting $\lambda \to 0$, we see that\quad $q_L(a) \ge \bra{b}{La} - f(b)$.\quad   If we now take the supremum over $b$ and use \eqref{FAT}, we see that $q_L(a) \ge f^@(a)$.\quad Consequently, $a \in \coinc[f^@]$. 
\end{proof}
The important thing about the next result is that $h$ is {\em not} required to be lower semicontinuous.
%:   \Thm{THREEthm} [The theorem of the three functions]
\begin{theorem}[The theorem of the three functions]\label{THREEthm}
Let $h \in \PC(E \times E^*)$,
%:   \eqref{THREE1}
\begin{equation}\label{THREE1}
h \ge q_L\ \on\ E \times E^*\hbox{ and }h^* \ge \qLt\ \on\  E^* \times E\dbs.
\end{equation}
Then $h^@ \ge q_L$ on $E \times E^*$ and $\coinc[h^@]$ is closed and quasidense.
\end{theorem}
%:
\begin{proof}
From \eqref{THREE1},\quad $h^@ = h^* \circ L \ge \qLt \circ L =  q_L$ on $E \times E^*$,\quad  as required.   From \Thm{HKFthm} or \Thm{HKF2thm} with $k = q_L$, there exists $f \in \PCLSC(E \times E^*)$ such that $f \ge q_L$ on $E \times E^*$ and $f^* = h^* \ge \qLt$ on $E^* \times E\dbs$, from which $f^@ = h^@ \ge q_L$ on $E \times E^*$.   Thus \Thm{FSTARthm} and  \Lem{Llem} imply that $\coinc[f]$ is quasidense and $\coinc[f] \subset \coinc[f^@]$.   Consequently, $\coinc[f^@]$ is quasidense.  Since $f^@ = h^@$ on $E \times E^*$, $\coinc[h^@]$ is quasidense also.  Since $q_L$ is continuous and $h^@$ is lower semicontinuous, $\coinc[h^@]$ is closed.
\end{proof}   
%:\Sec{EPIsec}
\section{The coincidence sets of partial episums}\label{EPIsec}
Let $E$ and $F$ be nonzero Banach spaces and $f, g \in \PCLSC(E \times F)$.   Then we define the functions $(f \epitwo g)$ and $(f \epione g)$ by
%:  \eqref{DD1}
\begin{equation}\label{DD1}
(f \epitwo g)(x,y) := \infn_{\eta \in F}\big[f(x,y - \eta) + g(x,\eta)]
\end{equation}
and
\begin{equation*}
(f \epione g)(x,y) := \infn_{\xi \in E}\big[f(x - \xi,y) + g(\xi,y)\big].
\end{equation*}
We substitute the symbol $\epetwo$ for $\epitwo$ and $\epeone$ for $\epione$ if the infimum is {\em exact}, that is to say, can be replaced by a minimum.   \Lem{SZlem} below first appeared in Simons--Z\u{a}linescu \cite[Section~4, pp.\ 8--10]{SZNZ}, and appeared subsequently in\break \cite[Section~16, pp.~67--69]{HBM}.   It was later generalized in \cite[Theorem 9, p.\ 882]{QUAD} and \cite[Corollary 5.4, pp.\ 121--122]{AST}.   We will be applying  \Lems{SZlem} and \ref{SZBISlem} below with $F := E^*$.   We define the projection maps $\pi_1$ and $\pi_2$ by $\pi_1(x,y) := x$ and $\pi_2(x,y) := y$ \big($(x,y) \in E \times F$\big).

\medbreak

%:   \Lem{SZlem}
\begin{lemma}\label{SZlem}
Let $f, g \in \PCLSC(E \times F)$, $f \epitwo g \in \PC(E \times F)$ and
\begin{equation*}
\ts\bigcupn_{\lambda > 0}\lambda\big[\pi_1\,\dom\,f - \pi_1\,\dom\,g\big]\ \hbox{be a closed subspace of}\ E.
\end{equation*}
Then\quad $(f \epitwo g)^* = f^* \epeone g^*\ \on\ E^* \times F^*$.
\end{lemma}
%:   \Thm{Dthm}
\begin{theorem}\label{Dthm}
Let $f, g \in \PCLSC(E \times E^*)$,  $f,g \ge q_L$ on $E \times E^*$,
%:  \eqref{D1}
\begin{equation}\label{D1}
\ts\bigcupn_{\lambda > 0}\lambda\big[\pi_1\,\dom\,f - \pi_1\,\dom\,g\big]\hbox{ be a closed subspace of }E,
\end{equation}
and $\coinc[f]$ and $\coinc[g]$ be quasidense.   Then $(f \epitwo g)^@ \ge q_L$ on $E \times E^*$, $\coinc[(f \epitwo g)^@]$ is closed and quasidense, and
%:  \eqref{DD2}
\begin{equation}\label{DD2}
\left.\begin{gathered}
(y,y^*) \in \coinc[(f \epitwo g)^@] \iff\\
\ex\ u^*,v^* \in E^*\ \st\\
(y,u^*) \in \coinc[f^@],\ (y,v^*) \in \coinc[g^@]\hbox{ and }u^* + v^* = y^*.
\end{gathered}\right\}
\end{equation}  
\end{theorem}
\begin{proof}
Let $h := f \epitwo g$.   Since $f,g \ge q_L$ on $E \times E^*$, for all $(x,x^*) \in E \times E^*$,
\begin{align*}
h(x,x^*)
&= \infn_{\xi^* \in E^*}\big[f(x,x^* - \xi^*) + g(x,\xi^*)\big]\\
&\ge \infn_{\xi^* \in E^*}\big[q_L(x,x^* - \xi^*) + q_L(x,\xi^*)\big]\\
&= \infn_{\xi^* \in E^*}\big[\bra{x}{x^* - \xi^*} + \bra{x}{\xi^*}\big] = \bra{x}{x^*} = q_L(x,x^*).
\end{align*}
From \eqref{D1}, $\pi_1\,\dom\,f \cap \pi_1\,\dom\,g \ne \emptyset$, and so there exist $x_0 \in E$, $y_0^* \in E^*$ and $z_0^* \in E^*$ such that $(x_0,y_0^*) \in \dom\,f$ and $(x_0,z_0^*) \in \dom\,g$.   It now follows from \eqref{DD1} that\quad $(f \epitwo g)(x_0,y_0^* + z_0^*) \le f(x_0,y_0^*) + g(x_0,z_0^*) < \infty$.\quad To sum up:
%:   \eqref{D2}
\begin{equation}\label{D2}  
h \in \PC(E \times E^*)\hbox{ and } h \ge q_L\ \on\ E \times E^*. 
\end{equation}
Note that we do not assert in \eqref{D2} that $h \in \PCLSC(E \times E^*)$.   Since $\coinc[f]$ and $\coinc[g]$ are quasidense, \Lem{FSTARlem} implies that 
%:   \eqref{D3}
\begin{equation}\label{D3}
f^* \ge \qLt\ \on\ E^* \times E\dbs \quand g^* \ge \qLt\ \on\ E^* \times E\dbs,
\end{equation}
from which
%:   \eqref{DD4}
\begin{equation}\label{DD4}
f^@ \ge q_L\ \on\ E \times E^* \quand g^@ \ge q_L\ \on\ E \times E^*.
\end{equation}

From \Lem{SZlem} and \eqref{D3}, for all $(y^*,y\dbs) \in E^* \times E\dbs$,
%:   \eqref{D4}
\begin{equation}\label{D4}
\left.
\begin{gathered}
h^*(y^*,y\dbs) = \minn_{z^* \in E^*}\big[{f}^*(y^* - z^*,y\dbs) + {g}^*(z^*,y\dbs)\big]\\
\ge \infn_{z^* \in E^*}\big[\bra{y^* - z^*}{y\dbs} + \bra{z^*}{y\dbs}\big]
= \bra{y^*}{y\dbs} = \qLt(y^*,y\dbs).
\end{gathered}
\right\}
\end{equation}
Thus $h^* \ge \qLt$ on  $E^* \times E\dbs$, and so \eqref{D2} and \Thm{THREEthm} imply that $h^@ \ge q_L$ on $E \times E^*$ and $\coinc[h^@]$ is closed and quasidense, as required.
\par
We now establish \eqref{DD2}.   If $(y,y^*) \in E \times E^*$ and we use \eqref{DD4} and  specialize \eqref{D4} to the case when $y\dbs = \wh y$, we obtain 
%:   \eqref{DD3}
\begin{equation}\label{DD3}
\left.
\begin{gathered}
h^@(y,y^*) = h^*(y^*,\wh y) = \minn_{z^* \in E^*}\big[f^@(y,y^* - z^*) + g^@(y,z^*)\big]\\
\ge \infn_{z^* \in E^*}\big[\bra{y}{y^* - z^*} + \bra{y}{z^*}\big] = \bra{y}{y^*} = q_L(y,y^*).
\end{gathered}
\right\}
\end{equation}
If $(y,y^*) \in \coinc[h^@]$ then this provides $v^* \in E^*$ such that
\begin{equation*}
f^@(y,y^* - v^*) + g^@(y,v^*) = \bra{y}{y^* - v^*} + \bra{y}{v^*}.
\end{equation*}
Let $u^* := y^* - v^*$. Then $u^* + v^* = y^*$ and $f^@(y,u^*) + g^@(y,v^*) = \bra{y}{u^*} + \bra{y}{v^*}$.\break   From \eqref{DD4}, $(y,u^*) \in \coinc[f^@]$ and $(y,v^*) \in \coinc[g^@]$.   This completes the proof of the implication ($\lr$) of \eqref{DD2}.  If, conversely, there exist $u^*,v^* \in E^*$ such that $(y,u^*) \in \coinc[f^@]$, $(y,v^*) \in \coinc[g^@]$ and $u^* + v^* = y^*$ then, from \eqref{DD3},
\begin{gather*}
h^@(y,y^*) \le f^@(y,u^*) + g^@(y,v^*) = \bra{y}{u^*} + \bra{y}{v^*} = \bra{y}{y^*}.
\end{gather*}
It now follows from \eqref{DD3} that $(y,y^*) \in \coinc[h^@]$.   This completes the proof of the implication ($\rl$) of \eqref{DD2}, and thus the proof of \Thm{Dthm}. 
\end{proof}
By interchanging the roles of $\epitwo$ and $\epione$ in the statement of \Lem{SZlem}, we can prove the following result:
%:   \Lem{SZBISlem}
\begin{lemma}\label{SZBISlem}
Let $f, g \in \PCLSC(E \times F)$, $f \epione g \in \PC(E \times F)$ and
\begin{equation*}
\ts\bigcupn_{\lambda > 0}\lambda\big[\pi_2\,\dom\,f - \pi_2\,\dom\,g\big]\ \hbox{be a closed subspace of}\ F.
\end{equation*}
Then\quad$(f \epione g)^* = f^* \epetwo g^*\ \on\ E^* \times F^*$.
\end{lemma}
%:   \Thm{Rthm} [Sum theorem with range constraints]
\begin{theorem}\label{Rthm}
Let $f, g \in \PCLSC(E \times E^*)$,  $f,g \ge q_L$ on $E \times E^*$,
%:  \eqref{R1}
\begin{equation}\label{R1}
\ts\bigcupn_{\lambda > 0}\lambda\big[\pi_2\,\dom\,f - \pi_2\,\dom\,g\big]\hbox{ be a closed subspace of }E^*,
\end{equation}
and $\coinc[f]$ and $\coinc[g]$ be quasidense.   Then $(f \epione g)^@ \ge q_L$ on $E \times E^*$, $\coinc[(f \epione g)^@]$ is closed and quasidense and,
%:  \eqref{R2}
\begin{equation}\label{R2}
\left.\begin{gathered}
(y,y^*) \in \coinc[(f \epione g)^@] \iff\\
\ex\ u\dbs,v\dbs \in E\dbs\ \st\\
(y^*,u\dbs) \in \dcoinc[f^*],\ (y^*,v\dbs) \in \dcoinc[g^*]\hbox{ and }u\dbs + v\dbs = \wh y.
\end{gathered}
\right\}
\end{equation}
\end{theorem}
\begin{proof}
Let $h := f \epione g$.   By interchanging the variables in the proofs already given of \eqref{D2} and \eqref{D3} in \Thm{Dthm}, we can prove that,
\begin{equation}\label{R3}
h \in \PC(E \times E^*)\hbox{ and } h \ge q_L\ \on\ E \times E^* 
\end{equation}
and
%:  \eqref{R4}
\begin{equation}\label{R4}
f^* \ge \qLt\ \on\ E^* \times E\dbs \quand g^* \ge \qLt\ \on\ E^* \times E\dbs.
\end{equation}
From \Lem{SZBISlem} and \eqref{R4}, for all $(y^*,y\dbs) \in E^* \times E\dbs$,
%:  \eqref{R5}
\begin{equation}\label{R5}
\left.
\begin{gathered}
h^*(y^*,y\dbs) = \minn_{z\dbs \in E\dbs}\big[{f}^*(y^*,y\dbs - z\dbs) + {g}^*(y^*,z\dbs)\\
\ge \infn_{z\dbs \in E\dbs}\big[\bra{y^*}{y\dbs - z\dbs} + \bra{y^*}{z\dbs}\big] = \bra{y^*}{y\dbs} = \qLt(y^*,y\dbs).
\end{gathered}
\right\}
\end{equation}
Thus $h^* \ge \qLt$ on  $E^* \times E\dbs$, and so \eqref{R3} and \Thm{THREEthm} imply that $h^@ \ge q_L$ on $E \times E^*$ and $\coinc[h^@]$ is closed and quasidense, as required.   If we now let $(y,y^*) \in E \times E^*$ and specialize \eqref{R5} to the case when $y\dbs = \wh y$, we obtain
%:   \eqref{R7}
\begin{equation}\label{R7}
\left.\begin{gathered}
h^@(y,y^*) = \minn_{z\dbs \in E\dbs}\big[f^*(y^*,\wh y - z\dbs) + g^*(y^*,z\dbs)\big]\\
\ge \infn_{z\dbs \in E\dbs}\big[\bra{y^*}{\wh y - z\dbs} + \bra{y^*}{z\dbs}\big] = \bra{y}{y^*} = q_L(y,y^*).
\end{gathered}\right\}
\end{equation}
We now establish \eqref{R2}.   If $(y,y^*) \in \coinc[h^@]$ then \eqref{R7} provides $v\dbs \in E\dbs$ such that
\begin{equation*}
f^*(y^*,\wh y - v\dbs) + g^*(y^*,v\dbs) = \bra{y^*}{\wh y - v\dbs} + \bra{y^*}{v\dbs}.
\end{equation*}
Let $u\dbs := \wh y - v\dbs$.   Then we have $u\dbs + v\dbs = \wh y$ and $f^*(y^*,u\dbs) + g^*(y^*,v\dbs) = \bra{y^*}{u\dbs} + \bra{y^*}{v\dbs}$.   From \eqref{R4}, $(y^*,u\dbs) \in \dcoinc[f^*]$ and $(y^*,v\dbs) \in \dcoinc[g^*]$.   This completes the proof of the implication ($\lr$) of \eqref{R2}.   If, conversely, there exist
$u\dbs,v\dbs \in E\dbs$ such that $(y^*,u\dbs) \in \dcoinc[f^*]$, $(y^*,v\dbs) \in \dcoinc[g^*]$ and $u\dbs + v\dbs = \wh y$ then, from \eqref{R7},
\begin{equation*}
h^@(y,y^*) \le f^*(y^*,u\dbs) + g^*(y^*,v\dbs) = \bra{y^*}{u\dbs} + \bra{y^*}{v\dbs} = \bra{y^*}{\wh y} = \bra{y}{y^*}.
\end{equation*}
It now follows from \eqref{R7} that $(y,y^*) \in \coinc[h^@]$.   This completes the proof of the implication ($\rl$) of \eqref{R2}, and thus the proof of \Thm{Rthm}. 
\end{proof}
%:\Sec{MONsec}{Monotone sets and multifunctions}
\section{Monotone sets and multifunctions}\label{MONsec}
Let $\emptyset \ne A \subset E \times E^*$.   It is easy to see that
%:   \eqref{QLMON} \eqref{LMON}
\begin{gather}
A\hbox{\em\ is monotone if, and only if, for all } a,b \in A,\ q_L(a - b) \ge 0\label{QLMON}\\
\hbox{\em\ if, and only if, } L(A)\hbox{\em\ is a monotone subset of }E^* \times E\dbs.\label{LMON}
\end{gather}
%
%:   \Thm{RLMAXthm} [Quasidensity and maximality]
\begin{theorem}[Quasidensity and maximality]\label{RLMAXthm}
Let $A$ be a closed, quasidense monotone subset of  $E \times E^*$.   Then $A$ is maximally monotone. 
\end{theorem}
\begin{proof}
Let $c \in E \times E^*$ and $A \cup \{c\}$ be monotone.   Let $\eps > 0$, and choose $a \in A$ so that $r_L(a - c) < \eps$.   Since $q_L(a - c) \ge 0$, it follows that
\begin{align*}
\half\|a - c\|^2 \le \half\|a - c\|^2 + q_L(a - c) = r_L(a - c) < \eps.
\end{align*}
Letting $\eps \to 0$ and using the fact that $A$ is closed, $c \in A$.
\end{proof}
The following important property of coincidence sets was first proved in Burachik--Svaiter, \cite[Theorem~3.1, pp. 2381--2382]{BS} and Penot, \cite[Proposition 4(h)$\lr$(a), pp. 860--861]{PENOT}.   Here, we give a short proof using the criterion for monotonicity that appeared in \eqref{QLMON}.
%:   \Lem{CONTlem}
\begin{lemma}\label{CONTlem}
Let $f \in \PC(E \times E^*)$ and $f \ge q_L$ on $E \times E^*$
.   Then $\coinc[f]$ is monotone.
\end{lemma}
\begin{proof} Let $a,b \in \coinc[f]$.   Then
\begin{align*}
\fourth q_L(a - b) &= \half q_L(a) + \half q_L(b) - \fourth q_L(a + b) = \half f(a) + \half f(b) - \fourth q_L(a + b)\\
&\ge f\big(\half(a + b)\big) - q_L\big(\half(a + b)\big) \ge 0.
\end{align*}
This establishes \eqref{QLMON} and completes the proof of \Lem{CONTlem}. 
\end{proof}
In order to simplify some notation in the sequel, if $S\colon\ E \toto E^*$, we will say that $S$ is {\em closed} if its graph, $G(S)$, is closed in $E \times E^*$, and we will say that $S$ is \emph{quasidense} if $G(S)$ is quasidense in $E \times E^*$.
\par
Our analysis depends on the following definition:
%:   \Def{THdef} [The definition of $\theta_S$]
\begin{definition}[The definition of $\theta_S$]\label{THdef}
Let $S\colon\ E \toto E^*$ be a monotone\break multifunction and $G(S) \ne \emptyset$.   We define the function $\theta_S \in \PCLSC(E^* \times E\dbs)$ by $\theta_S := \Theta_{G(S)}$.   (See \Def{THAdef}.)   Explicitly:
%:  \eqref{TH1}
\begin{equation}\label{TH1}
\all\ c^* \in E^* \times E\dbs,\quad \theta_S(c^*) := \supn_{G(S)}[c^* - q_L].
\end{equation}
In longhand, for all $(y^*,y\dbs) \in E^* \times E\dbs$
%:   \eqref{THLONG}
\begin{equation}\label{THLONG}
\theta_S(y^*,y\dbs) := \supn_{(s,s^*) \in G(S)}\big[\bra{s}{y^*} + \bra{s^*}{y\dbs} - \bra{s}{s^*}\big].
\end{equation}
\end{definition}
We now show how $\theta_S$ determines the {\em Fitzpatrick function}, $\varphi_S$, that acts on $E \times E^*$ (rather than on $E^* \times E\dbs$).
%:   \Def{PHdef} [The definition of $\varphi_S$]
\begin{definition}[The definition of $\varphi_S$]\label{PHdef}
Let $S\colon\ E \toto E^*$ be a monotone multifunction and $G(S) \ne \emptyset$.   We define the function $\varphi_S \in \PCLSC(E \times E^*)$ by
%:  \eqref{TH2}
\begin{equation}\label{TH2}
\varphi_S = \theta_S \circ L.
\end{equation}
Explicitly,
\begin{align}
%:   \eqref{PH1}
\all\ b \in E \times E^*,\quad\varphi_S(b) &:= \supn_{G(S)}[Lb - q_L]\label{PH1}\\
%:   \eqref{PH2}
&= q_L(b) - \inf q_L\big(G(S) - b\big).\label{PH2}
\end{align}
In longhand, for all $(x,x^*) \in E \times E^*$,
%:   \eqref{PH5}
\begin{equation}\label{PH5}
\varphi_S(x,x^*) := \supn_{(s,s^*) \in G(S)}\big[\bra{s}{x^*} + \bra{x}{s^*} - \bra{s}{s^*}\big].
\end{equation}
\end{definition}
%:   \Rem{FDEFrem}
\begin{remark}\label{FDEFrem}
The Fitzpatrick function was originally introduced in the\break Banach space setting in \cite[(1988)]{FITZ}, but lay dormant until it was rediscovered by  Mart\'\i nez-Legaz and Th\'era in \cite[(2001)]{MLT}.  It had been previously considered in the finite--dimensional setting by Krylov in \cite[(1982)]{KRYLOV}. 
The generalization of the Fitzpatrick function to {\em Banach SN spaces} can  be found in  \cite[Definition 6.2, p.\ 1029]{PARTONE}.
\end{remark}
%:   \Lem{PHlem}
\begin{lemma}\label{PHlem}
Let $S\colon\ E \toto E^*$ be maximally monotone.   Then:
%:  \eqref{PH3}
\begin{equation}\label{PH3}
\varphi_S \in \PCLSC(E \times E^*),\ \varphi_S \ge q_L\ \on\ E \times E^*\quand \coinc[\varphi_S] = G(S).
\end{equation}
\end{lemma}
\begin{proof}
If $b \in E \times E^*$ and $\varphi_S(b) \le q_L(b)$ then \eqref{PH2} gives $\inf q_L\big(G(S) - b\big) \ge 0$.   From the maximality, $b \in G(S)$ and so we derive from the monotonicity that $\inf q_L\big(G(S) - b\big) = 0$, from which $\varphi_S(b) = q_L(b)$.   Since $\varphi_S$ is obviously convex and lower semicontinuous, this completes the proof of \eqref{PH3}.
\end{proof}
We now come to the ``${\varphi_S}^*$ criterion'' for a maximally monotone set to be quasidense.
%:   \Thm{PHthm}
\begin{theorem}\label{PHthm}
Let $S\colon\ E \toto E^*$ be maximally monotone.   Then:
%:  \eqref{PH4}
\begin{equation}\label{PH4}
S\hbox{ is quasidense } \iff {\varphi_S}^* \ge \qLt\ \on\ E^* \times E\dbs.
\end{equation}
\end{theorem}
\begin{proof}
This is immediate from \eqref{PH3} and \Thm{FSTARthm}.
\end{proof}   
%:   \Cor{SURJcor} [First partial converse to \Thm{RLMAXthm}]
\begin{corollary}[First partial converse to \Thm{RLMAXthm}]\label{SURJcor}
Let $S\colon\ E \toto E^*$ be maximally monotone and surjective.   Then $S$ is quasidense.
\end{corollary}
\begin{proof}
Suppose that $(y^*,y\dbs) \in E^* \times E\dbs$.   Let $x \in S^{-1}y^*$.   Then, from \eqref{PH3},
\begin{align*}
{\varphi_S}^*(y^*,y\dbs) &\ge \bra{x}{y^*} + \bra{y^*}{y\dbs} - \varphi_S(x,y^*)\\
&= \bra{x}{y^*} + \bra{y^*}{y\dbs} - \bra{x}{y^*} = \bra{y^*}{y\dbs} = \qLt(y^*,y\dbs).
\end{align*}   
It now follows from \eqref{PH4} that $S$ is quasidense.
\end{proof}
\begin{remark}
Once one knows the (highly nontrivial) result that a maximally\break monotone multifunction is quasidense if, and only if, it is {\em of type (FP), {\em or} locally maximally monotone}, see \cite[Theorem 10.3, p.\ 21]{PARTTWO}, then \Cor{SURJcor} follows from Fitzpatrick--Phelps, \cite[Theorem 3.7, pp.\ 67--68]{FITZTWO}.
\end{remark}
In \Thm{THthm}, we will give the ``$\theta_S$ criterion'' for a maximally monotone set to be quasidense.   We start with a preliminary lemma of independent interest,   which will be used in \Cor{PHIVcor}. \Lem{THlem} raises the following problem:
%:   \Prob{THprob}
\begin{problem}\label{THprob}
Is there a maximally monotone multifunction $S\colon\ E \toto E^*$ such that ${\varphi_S}^* \ne \theta_S$?
\end{problem}
%:   \Lem{THlem}
\begin{lemma}\label{THlem}
Let $S\colon\ E \toto E^*$ be maximally monotone.   Then:
%:  \eqref{TH3}
\begin{equation}\label{TH3}
{\varphi_S}^* \ge \theta_S\ \on\ E^* \times E\dbs.
\end{equation}
If, further,
%:  \eqref{TH5}
\begin{equation}\label{TH5}
\dom\,\varphi_S \subset G(S)
\end{equation}
then
%:  \eqref{TH4}
\begin{equation}\label{TH4}
{\varphi_S}^* = \theta_S\ \on\ E^* \times E\dbs.
\end{equation}
\end{lemma}
\begin{proof}
Let $c^* \in E^* \times E\dbs$.   From \eqref{PH3} and \eqref{TH1},
\begin{equation*}
{\varphi_S}^*(c^*) = \supn_{E \times E^*}[{c^*} - \varphi_S]
\ge \supn_{G(S)}[{c^*} - \varphi_S] = \supn_{G(S)}[{c^*} - q_L] = \theta_S(c^*),
\end{equation*}
which gives \eqref{TH3}.   Now suppose that \eqref{TH5} is satisfied.      If $b \in E \times E^* \setminus \dom\,\varphi_S$ then $\bra{b}{c^*} - \varphi_S(b) = -\infty \le \theta_S(c^*)$.   If, on the other hand, $b \in \dom\,\varphi_S$ then \eqref{TH5} implies that $b \in G(S)$, and so \eqref{PH3} gives $\varphi_S(b) = q_L(b)$.   Thus, using \eqref{TH1}, $\bra{b}{c^*} - \varphi_S(b) = \bra{b}{c^*} - q_L(b) \le \theta_S(c^*)$.   Combining these two observations, we see that,
\begin{equation*}
\all\ b \in E \times E^*,\quad \bra{b}{c^*} - \varphi_S(b) \le \theta_S(c^*).
\end{equation*}
Taking the supremum over $b \in E \times E^*$, ${\varphi_S}^*(c^*) \le \theta_S(c^*)$.   Thus ${\varphi_S}^* \le \theta_S$ on $E^* \times E\dbs$, and \eqref{TH4} follows from  \eqref{TH3}.      
\end{proof}
%
%:   \Thm{THthm}
\begin{theorem}\label{THthm}
Let $S\colon\ E \toto E^*$ be maximally monotone.   Then:
%:  \eqref{PHIA5}
\begin{equation}\label{PHIA5}
S\hbox{ is quasidense } \iff \theta_S \ge \qLt\ \on\ E^* \times E\dbs.
\end{equation}
\end{theorem}
\begin{proof}
If $S$ is quasidense then $G(S)$ is a quasidense subset of $E \times E^*$ and so, from \Cor{THAcor}, $\theta_S = \Theta_{G(S)} \ge \qLt$ on $E^* \times E\dbs$.   If, conversely,  $\theta_S \ge \qLt$ on $E^* \times E\dbs$ then, from \eqref{TH3}, ${\varphi_S}^* \ge \qLt$ on $E^* \times E\dbs$, and it follows from \Thm{PHthm} that $S$ is quasidense.  
\end{proof}
%:   \Cor{CONVcor} [Second partial converse to \Thm{RLMAXthm}]
\begin{corollary}[Second partial converse to \Thm{RLMAXthm}]\label{CONVcor}
Let $E$ be reflexive and $S\colon\ E \toto E^*$ be maximally monotone.   Then $S$ is quasidense.
\end{corollary}
\begin{proof}
Suppose that $(y^*,y\dbs) \in E^* \times E\dbs$.   Choose $y \in E$ such that $\wh y = y\dbs$.   Then $(y^*,y\dbs) = (y^*,\wh y) = L(y,y^*)$ and so, from \eqref{TH2} and \eqref{PH3},
\begin{equation*}
\theta_S(y^*,y\dbs) = \theta_S\circ L(y,y^*) = \varphi_S(y,y^*) \ge q_L(y,y^*) = \qLt(y^*,y\dbs).
\end{equation*}
It now follows from \Thm{THthm} that $S$ is quasidense.
\end{proof}
We end this section by giving a result in \Thm{COINCthm} that will be used in our discussion of the Fitzpatrick extension in \Sec{FITZEXTsec}.   We start with a preliminary lemma.
%:   \Lem{PSlem} [Various properties of $\varphi_S$]
\begin{lemma}\label{PSlem}
Let $S\colon\ E \toto E^*$ be maximally monotone.   Then:
\begin{gather}
%:  \eqref{PS2}
{\varphi_S}^@ \ge \varphi_S\  \ge q_L\ \on\ E \times E^*\quand \coinc[{\varphi_S}^@] = G(S).\label{PS2}\\
%:  \eqref{PS4}
{\theta_S}^@ \ge {\varphi_S}^* \ge \theta_S\ \on\ E^* \times E\dbs.\label{PS4}
\end{gather}
\end{lemma}
\begin{proof}
It follows by composing \eqref{TH3} with $L$ and using \Def{FATdef} and \eqref{TH2} that ${\varphi_S}^@ \ge \varphi_S$ on $E \times E^*$.   Furthermore, \eqref{PH3} implies that $\varphi_S \ge q_L$ on $E \times E^*$ and $G(S) = \coinc[\varphi_S] \supset \coinc[{\varphi_S}^@]$.   \Lem{Llem} implies that $\coinc[\varphi_S] \subset \coinc[{\varphi_S}^@]$, which completes the proof of \eqref{PS2}.
\par
For all $c^* \in E^*\times E\dbs$, ${\theta_S}^@(c^*) = \supn_{E^* \times E\dbs}[{\LT c^*} - \theta_S]$.   Thus, from \eqref{TH2},  
\begin{align*}
{\theta_S}^@(c^*)
&\ge \supn_{b \in E \times E^*}\big[\bra{Lb}{\LT c^*} - \theta_S(Lb)\big]\\
&= \supn_{b \in E \times E^*}\big[\bra{b}{c^*} - \varphi_S(b)\big] = {\varphi_S}^*(c^*),
\end{align*}
which gives the first inequality in \eqref{PS4}, and the second inequality in \eqref{PS4} has already been established in \eqref{TH3}.      
\end{proof}
%:   \Thm{COINCthm}
\begin{theorem}\label{COINCthm}
Let $S\colon\ E \toto E^*$ be maximally monotone and quasidense.   Then\quad $\dcoinc[{\theta_S}] = \dcoinc[{\varphi_S}^*] = \dcoinc[{\theta_S}^@]$.
\end{theorem}
\begin{proof}
From \eqref{PS4} and \eqref{PHIA5},\quad ${\theta_S}^@ \ge {\varphi_S}^* \ge \theta_S \ge \qLt$ on $E^* \times E\dbs$.\quad It follows that\quad $\dcoinc[{\theta_S}^@] \subset \dcoinc[{\varphi_S}^*] \subset \dcoinc[{\theta_S}]$.\quad  However, if we apply\break \Lem{Llem} (to $E^* \times E\dbs$ instead of $E \times E^*$), we see that $\dcoinc[\theta_S] \subset \dcoinc[{\theta_S}^@$].   This gives the desired result.
\end{proof}
\begin{problem}
\Thm{COINCthm} leads to the question:  {\em if $S$ is maximally monotone and ${\theta_S}^@ \ge \qLt$ on $E^* \times E\dbs$ then is $S$ necessarily quasidense?} 
\end{problem}
%:\Sec{DSUMSsec}
\section{Sum theorem with domain constraints}\label{DSUMSsec}
\begin{notation}
Let $S\colon\ E \toto E^*$.   In what follows, we write
\begin{equation*}
D(S) := \big\{x \in E\colon\ Sx \ne \emptyset\big\} = \pi_1G(S)\hbox{ and }R(S) := \ts\bigcup_{x \in E}Sx = \pi_2G(S).
\end{equation*}
\end{notation}
We will use the following computational rules in the sequel:
%:   \Lem{PHISlem}
\begin{lemma}\label{PHISlem}
Let $S\colon\ E \toto E^*$ be closed, quasidense and monotone.   Then
\begin{align}
D(S) \subset \pi_1\dom\,\varphi_S &\quand R(S) \subset \pi_2\dom\,\varphi_S.\label{PHIS2}
\end{align}
\end{lemma}
%:
\begin{proof}
This is immediate from \Thm{RLMAXthm} and \eqref{PH3}.
\end{proof}

%:   \Thm{STDthm} [Sum theorem with domain constraints]
\begin{theorem}[Sum theorem with domain constraints]\label{STDthm}
Let $S,T\colon\ E \toto E^*$ be closed, quasidense and monotone.   Then {\rm(a)$\lr$(b)$\lr$(c)$\lr$(d)}:
\par
\noindent
{\rm(a)}\enspace $D(S) \cap \intr\,D(T) \ne \emptyset$ or $\intr\,D(S) \cap D(T) \ne \emptyset$.
\par
\noindent
{\rm(b)}\enspace $\ts\bigcupn_{\lambda > 0}\lambda\big[D(S) - D(T)\big] = E$.
\par
\noindent
{\rm(c)}\enspace $\ts\bigcupn_{\lambda > 0}\lambda\big[\pi_1\,\dom\,\varphi_S - \pi_1\,\dom\,\varphi_T\big]$\quad is a closed subspace of $E$.
\par
\noindent
{\rm(d)}\enspace $S + T$\quad is closed, quasidense and monotone.
\end{theorem}
\begin{proof}
It is immediate from \eqref{PHIS2} that (a)$\lr$(b)$\lr$(c).   Now suppose that (c) is satisfied.   From \Thm{RLMAXthm}, $S$ and $T$ are maximally monotone, and so \eqref{PH3} and \eqref{PS2} imply that $\varphi_S, \varphi_T \in \PCLSC(E \times E^*)$, $\varphi_S, \varphi_T \ge q_L$ on $E \times E^*$, $\coinc[\varphi_S] = \coinc[{\varphi_S}^@] = G(S)$ and $\coinc[\varphi_T] = \coinc[{\varphi_T}^@] = G(T)$, and we can apply \Thm{Dthm} with $f := \varphi_S$ and $g := \varphi_T$.

Thus $(\varphi_S \epitwo \varphi_T)^@ \ge q_L$ on $E \times E^*$, $\coinc[(\varphi_S \epitwo \varphi_T)^@]$ is closed and quasidense, and $(y,y^*) \in \coinc[(\varphi_S \epitwo \varphi_T)^@]$ if, and only if, there exist $u^*,v^* \in E^*$ such that $(y,u^*) \in G(S)$, $(y,v^*) \in G(T)$ and $u^* + v^* = y^*$.   This is exactly equivalent to the statement that $(y,y^*) \in G(S + T)$.   Finally, it is obvious that $S + T$ is monotone. 
\end{proof}
%:   \Rem{VZrem}
\begin{remark}\label{VZrem}
\Thm{STDthm} above has applications to the classification of maximally monotone multifunctions.   See \cite[Theorems~7.2 and 8.1]{PARTTWO}.   \Thm{STDthm} can also be deduced from Voisei--Z\u{a}linescu \cite[Corollary~3.5,\ p.\ 1024]{VZ}.
\end{remark}
%:\Sec{FITZEXTsec}
\section{The Fitzpatrick extension}\label{FITZEXTsec}
%:   \Def{FITZdef} [The Fitzpatrick extension]
\begin{definition}[The Fitzpatrick extension]\label{FITZdef}
Let $S\colon\ E \toto E^*$ be a closed quasidense monotone multifunction.   We now introduce the {\em Fitzpatrick extension}, $S^\F\colon\ E^* \toto E\dbs$, of $S$.   From \Thm{RLMAXthm} and \eqref{PH3}, $\coinc[\varphi_S] = G(S)$, and so we see from \Thm{FSTARthm} that ${\varphi_S}^* \ge \qLt$ on $E^* \times E\dbs$. 
Using our current notation, the multifunction $S^\F$ was defined in \cite[Definition 5.1]{PARTTWO} by
%   \eqref{PHSTCRIT}
\begin{equation}\label{PHSTCRIT}
G(S^\F) := \dcoinc[{\varphi_S}^*].
\end{equation}
\big(There is a more abstract version of this in  \cite[Definition 8.5, p.\ 1037]{PARTONE}.\big)   From \Thm{COINCthm}, we can also write
%:   \eqref{THCRIT}
\begin{equation}\label{THCRIT}
G(S^\F) = \dcoinc[{\theta_S}] = \dcoinc[{\theta_S}^@].
\end{equation}
The word {\em extension} is justified by the fact that $L(a) \in G(S^\F) \iff a \in G(S)$.   Indeed, from \eqref{THCRIT}, \eqref{TH2} and \eqref{PH3},
%:   \eqref{EXT1}
\begin{equation}\label{EXT1}
\left.\begin{gathered}
L(a) \in G(S^\F) \iff \theta_S\big(L(a)\big) = \qLt\big(L(a)\big)\\
\iff \varphi_S(a) = q_L(a) \iff a \in G(S).
\end{gathered}
\right\}
\end{equation}
\end{definition}
%:   \Thm{AFMAXthm}
\begin{theorem}\label{AFMAXthm}
Let $S\colon\ E \toto E^*$ be closed, quasidense and monotone.   Then $S^\F$ is maximally monotone.
\end{theorem}
\begin{proof}
From \Lem{CONTlem} (applied to the function ${\varphi_S}^*$ on  $E^* \times E\dbs$), $S^\F$ is monotone.   Now let $c^* \in E^* \times E\dbs$ and, for all $a^* \in G(S^\F)$, $\qLt(c^* - a^*)\ge 0$.   From  \eqref{EXT1}, for all $a \in G(S)$, $\qLt\big(c^* - L(a)\big)\ge 0$.   Now \eqref{QD2} gives $\qLt\big(c^* - L(a)\big) = \qLt(c^*) - \bra{a}{c^*} + q_L(a)$ and so, for all $a \in G(S)$, $\qLt(c^*) \ge \bra{a}{c^*} - q_L(a)$.   Taking the supremum over $a$ and using \eqref{TH1}, $\qLt(c^*) \ge \theta_S(c^*)$.   From \Thm{THthm}, $\theta_S(c^*) = \qLt(c^*)$, and so $c^* \in \dcoinc[\theta_S]$.   Thus, from \eqref{THCRIT}, $c^* \in G(S^\F)$.   This completes the proof of the maximal monotonicity of $S^\F$.   
\end{proof}
\begin{remark}
It is interesting to speculate (see \cite[Problem 12.7, p.\ 1047]{PARTONE}) whether $S^\F$ is actually quasidense.   We shall see in \Ex{TAILex}, \Thms{SFTthm}(b) and \ref{SPECTthm} that this is not generally the case.   However, it is the case in one important situation.   We observed in \Ex{SUBex} that if $f\colon\ E \to \rbar$ is proper, convex and lower semicontinuous then $\partial f\colon\ E \toto E^*$ is quasidense.   However, it was shown in \cite[Theorem 5.7]{PARTTWO} that $(\partial f)^\F = \partial(f^*)$, so the multifunction $(\partial f)^\F\colon\ E^* \toto E\dbs$ is quasidense.
\end{remark}
%:   \Rem{GOSSrem}
\begin{remark}\label{GOSSrem}
It follows from \eqref{THCRIT} that $y\dbs \in S^\F(y^*)$ exactly when $(y\dbs,y^*)$ is in the {\em Gossez extension} of $G(S)$ \big(see \cite[Lemma~2.1, p.\ 275]{GOSSEZ}\big).
\end{remark}
Our next result gives a situation in which we can obtain an explicit description of $S^\F$, as well as inverse of the operation $S \mapsto S^\F$.   \Thm{TSthm} is an extension to the nonlinear case of \cite[Theorem 2.1, pp.\ 297--298]{BUS12}.   It will be important in our construction of examples.
%:   \Thm{TSthm}
\begin{theorem}\label{TSthm}
Let $T\colon\ E^* \toto E\dbs$ and $R(T) \subset \wh E$.   Let $S = G^{-1}L^{-1}G(T)$, {\em i.e.}, $S\colon\ E \toto E^*$ is defined by $G(S) = L^{-1}G(T)$.  Then:
\par\noindent
{\rm(a)}\enspace $G(T) \subset L\big(G(S)\big)$.   {\em(The opposite inclusion is trivially true.)}
\par\noindent
{\rm(b)}\enspace  Suppose in addition that $T$ is maximally monotone.   Then $S$ is maximally monotone.
\par\noindent
{\rm(c)}\enspace  Suppose in addition that $T$ is maximally monotone and $D(T) = E^*$.   Then $S$ is maximally monotone and quasidense, and $S^\F = T$.   {\em Put another way, for multifunctions like $T$}, $G^{-1}L^{-1}G$ is the inverse of $\cdot^\F$.
\end{theorem}
\begin{proof}
(a)\enspace Let $(y^*,y\dbs) \in G(T)$.   Since $R(T) \subset \wh E$, there exists $y \in E$ such that $y\dbs = \wh y$.   But then $(y^*,y\dbs) = L(y,y^*)$, and so $(y,y^*) \in L^{-1}\{(y^*,y\dbs)\} \subset L^{-1}G(T) = G(S)$, from which $(y^*,y\dbs) = L(y,y^*) \in L\big(G(S)\big)$.
\par
(b)\enspace Now let $b_1,b_2 \in G(S)$.   Then $Lb_1,Lb_2 \in G(T)$, and so $\qLt(Lb_1 - Lb_2) \ge 0$.   Equivalently, $q_L(b_1 - b_2) \ge 0$.   Thus $S$ is monotone.   We now prove that $S$ is maximally monotone.   To this end, let $c \in E \times E^*$ and $\inf q_L\big(G(S) - c\big) \ge 0$.   Equivalently, $\inf\qLt\big(L\big(G(S)\big) - Lc\big) \ge 0$.   From (a),  $\inf\qLt\big(G(T) - Lc\big) \ge 0$.  The maximal monotonicity of $T$ now implies that $Lc \in L\big(G(S)\big)$.   Since $L$ is injective, $c \in G(S)$.   Thus $S$ is maximally monotone.
\par
(c)\enspace Let $y^* \in E^* = D(T)$.   Arguing as in (a), there exist $y\dbs \in E\dbs$ and $y \in E$ such that $L(y,y^*) = (y^*,y\dbs) \in L\big(G(S)\big)$.   Since $L$ is injective, $(y,y^*) \in G(S)$.   Thus $R(S) = E^*$, and the quasidensity of $S$ follows from \Cor{SURJcor}.   \eqref{EXT1} and (a) now imply that  that $G(S^\F) \supset L\big(G(S)\big) \supset G(T)$, and the assumed maximal monotonicity of $T$ now gives $S^\F = T$, as required.
\end{proof}
The following result appears in Phelps--Simons, \cite[Corollary 2.6, p.\ 306]{PS}.   We do not know the original source of the result.   We almost certainly learned about it by personal communication with Robert Phelps.   We give a proof for completeness.
\begin{fact}\label{FOLKLORE}
Let $T\colon\ E \to E^*$ be monotone and linear.   Then $T$ is maximally monotone.
\end{fact}
\begin{proof}
Let $(y,y^*) \in E \times E^*$ and, for all $x \in E$, $\bra{x - y}{Tx - y^*} \ge 0$.   We first prove that, for all $z \in E$ and for all $\lambda \in \RR$,
%:   \eqref{FOLK1}
\begin{equation}\label{FOLK1}
\lambda\bra{z}{Ty - y^*} + \lambda^2\bra{z}{Tz} \ge 0.
\end{equation}
To this end, let $z \in E$ and $\lambda \in \RR$.   By direct computation,
\begin{equation*}
\lambda\bra{z}{Ty - y^*} + \lambda^2\bra{z}{Tz} = \bra{y + \lambda z - y}{T(y + \lambda z) - y^*}.
\end{equation*}
\eqref{FOLK1} now follows from our assumption, with $x = y + \lambda z$.   From \eqref{FOLK1}, for all $z \in E$, the quadratic expression $\lambda \mapsto \lambda\bra{z}{Ty - y^*} + \lambda^2\bra{z}{Tz}$ attains a minimum at $\lambda = 0$ so, from elementary calculus, for all $z \in E$, $\bra{z}{Ty - y^*} = 0$.   Consequently, $Ty - y^* = 0 \in E^*$.   Thus $y^* = Ty$.   This completes the proof of the maximal monotonicity of $T$.
\end{proof}
\Thm{LINVWthm} will be applied in \Ex{TAILex} and \Thm{SFTthm}.
%:   \Thm{LINVWthm}
\begin{theorem}\label{LINVWthm}
Let $T\colon\ E^* \to E\dbs$ be a monotone linear map and $R(T) \subset \wh E$.   Let $S = G^{-1}L^{-1}G(T)$.   Then $S$ is maximally monotone and quasidense, and $S^\F = T$.
\end{theorem}
\begin{proof}
Fact \ref{FOLKLORE} (with $E$ replaced by $E^*$) implies that $T$ is maximally monotone.   The result now follows from \Thm{TSthm}(c).
\end{proof}
%:   \Ex{TAILex}
\begin{example}\label{TAILex}
Let $E = \ell_1$, and define $T\colon\ \ell_1 \to \ell_\infty = {\ell_1}^*$ by $(Tx)_n = \sum_{k \ge n} x_k$.   $T$ is the ``tail operator''.   Let $S = G^{-1}L^{-1}G(T)$.   It was  proved in \cite[Example 7.10, pp.\ 1034--1035]{PARTONE} that $T$ is not quasidense.   Thus, from  \Thms{AFMAXthm} and \ref{LINVWthm}, $S$ is maximally monotone and quasidense, but $S^\F$ is maximally monotone and not quasidense.   This example answers in the negative the question posed in \cite[Problem 12.7, p.\ 1047]{PARTONE} as to whether the Fitzpatrick extension of a quasidense maximally monotone multifunction is necessarily quasidense.   $S$ can be represented in matrix form by
\begin{equation*}
\left(\begin{matrix}
(Sx)_1\\(Sx)_2\\(Sx)_3\\(Sx)_4\\(Sx)_5\\\vdots
\end{matrix}
\right)
= 
\left(\begin{matrix}
1&-1&0&0&0&\cdots\\
0&1&-1&0&0&\cdots\\
0&0&1&-1&0&\cdots\\
0&0&0&1&-1&\cdots\\
0&0&0&0&1&\cdots\\
\vdots&\vdots&\vdots&\vdots&\vdots&\ddots
\end{matrix}
\right)
\left(\begin{matrix}
x_1\\x_2\\x_3\\x_4\\x_5\\\vdots
\end{matrix}
\right),
\end{equation*}
and $D(S) =  \big\{x \in c_0\colon\ \tsum_{i = 1}^\infty|x_i - x_{i + 1}| < \infty\big\}$.
\end{example}
%:\Sec{RSUMSsec}
\section{Sum theorem with range constraints}\label{RSUMSsec}
\Thm{STRthm} below has applications to the classification of maximally monotone multifunctions.   See \cite[Theorems~8.2 and 10.3]{PARTTWO}.
%
%:   \Thm{STRthm} [Sum theorem with range constraints]
\begin{theorem}[Sum theorem with range constraints]\label{STRthm}
Let $S,T\colon\ E \toto E^*$ be closed, quasidense and monotone.      Then {\rm(a)$\lr$(b)$\lr$(c)$\lr$(d)}:
\par
\noindent
{\rm(a)}\enspace $R(S) \cap \intr\,R(T) \ne \emptyset$ or $\intr\,R(S) \cap R(T) \ne \emptyset$.
\par
\noindent
{\rm(b)}\enspace $\ts\bigcupn_{\lambda > 0}\lambda\big[R(S) - R(T)\big] = E^*$.
\par
\noindent
{\rm(c)}\enspace $\ts\bigcupn_{\lambda > 0}\lambda\big[\pi_2\,\dom\,\varphi_S - \pi_2\,\dom\,\varphi_T\big]$\quad is a closed subspace of $E^*$.
\par
\noindent
{\rm(d)}\enspace The multifunction $E \toto E^*$ defined by $y \mapsto (S^\F + T^\F)^{-1}(\wh y)$ is closed,\break quasidense and monotone.
\par
\noindent
{\rm(e)}\enspace If, further, $R(T^\F) \subset \wh E$, then the {\em parallel sum} $(S^{-1} + T^{-1})^{-1}$ is closed, monotone and  quasidense.
\end{theorem}
\begin{proof}
It is immediate \big(using \eqref{PHIS2}\big) that (a)$\lr$(b)$\lr$(c).   Now suppose that (c) is satisfied.    From \Thm{RLMAXthm}, $S$ and $T$ are maximally monotone, and so \eqref{PH3} implies that $\varphi_S, \varphi_T \in \PCLSC(E \times E^*)$, $\varphi_S, \varphi_T \ge q_L$ on $E \times E^*$, $\coinc[\varphi_S] = G(S)$ and $\coinc[\varphi_T] = G(T)$, and we can apply \Thm{Rthm} with $f := \varphi_S$ and $g := \varphi_T$.   Thus $(\varphi_S \epione \varphi_T)^@ \ge q_L$ on $E \times E^*$, $\coinc[(\varphi_S \epione \varphi_T)^@]$ is closed and quasidense, and $(y,y^*) \in \coinc[(\varphi_S \epione \varphi_T)^@]$ if, and only if, there exist $u\dbs,v\dbs \in E\dbs$ such that
%:   \eqref{DCON1}
\begin{equation}\label{DCON1}
(y^*,u\dbs) \in \dcoinc[{\varphi_S}^*],\ (y^*,v\dbs) \in \dcoinc[{\varphi_T}^*]\hbox{ and }u\dbs + v\dbs = \wh y.
\end{equation}
From \eqref{PHSTCRIT}, this is equivalent to the statement: ``$u\dbs \in S^\F(y^*)$, $v\dbs \in T^\F(y^*)$ and $u\dbs + v\dbs = \wh y$\,'', that is to say, ``$\wh y \in (S^\F + T^\F)(y^*)$''.   This gives (d).
\par
(e)\enspace Now suppose that $R(T^\F) \subset \wh E$ and $(y,y^*) \in \coinc[(\varphi_S \epione \varphi_T)^@]$.   Then the element $v\dbs$ in \eqref{DCON1} is actually in $\wh E$, and so there exists $v \in E$ such that\break $\wh v = v\dbs \in T^\F(y^*)$.   \eqref{EXT1} now implies that $(v,y^*) \in G(T)$, that is to say\break $v \in T^{-1}y^*$.   From \eqref{DCON1} again, $u\dbs = \wh{y - v}$, and a repetition of the argument above gives $y - v \in S^{-1}y^*$. Consequently, we have  $y = v + (y - v) \in (S^{-1} + T^{-1})y^*$, that is to say $y^* \in (S^{-1} + T^{-1})^{-1}y$.   Thus we have proved that $\coinc[(\varphi_S \epione \varphi_T)^@] \subset G\big((S^{-1} + T^{-1})^{-1}\big)$   On the other hand, from  \eqref{EXT1} and \eqref{DCON1}, we always have $G\big((S^{-1} + T^{-1})^{-1}\big) \subset \coinc[(\varphi_S \epione \varphi_T)^@]$, completing the proof of (e).
\end{proof}
%:\Sec{ANOTHERsec}
\section{Another maximally monotone non--quasidense multifunction}\label{ANOTHERsec}
In Bueno--Svaiter, \cite[Proposition 1, pp.\ 84--85]{BUS13} an example is given of a maximally monotone skew linear operator from a subspace of $c_0$ into  $\ell_1$ which is maximally monotone but not {\em of type (D),} thus answering in the negative a conjecture of J. Borwein.  As observed in \cite[Remark 10.4, pp.\ 21--22]{PARTTWO}, a maximally monotone multifunction is of type (D) if, and only if, it is quasidense, so the Bueno--Svaiter example provides a maximally monotone non--quasidense multifunction on $c_0$.   In this section, we discuss a slight modification of this multifunction. Ironically, it is easier to establish the non--quasidensity than the maximal monotonicity.
%:   \Def{Qdef}
\begin{definition}\label{Qdef}
If $(x_n)$ is a real sequence such that $\tsum_{k = 1}^\infty x_k$ is convergent, we define the {\em tail sequence} of $x$, $(t(x)_n)$, by, for all $n \ge 1$, $t(x)_n = \sum_{k = n}^\infty x_k$.   Clearly
%:   \eqref{Q1}
\begin{equation*}
t(x) \in c_0\hbox{\quad and,\quad for all }j \ge 1,\quad x_j = t(x)_{j} - t(x)_{j + 1}. 
\end{equation*}
Let 
%:   \eqref{S0}
\begin{equation}\label{S0}
K:=  \big\{x = (x_i)_{i \ge 1}\colon\ \tsum_{i = 1}^\infty x_i = 0\quand \tsum_{p = 1}^\infty|t(x)_{p} + t(x)_{p + 1}| < \infty\big\}.
\end{equation}
$K$ is a vector subspace of $c_0$. Let $x \in K$.  For all $j \ge 1$, let
%:   \eqref{S1}
\begin{equation}\label{S1}
(Sx)_j := -t(x)_{j} - t(x)_{j + 1}.
\end{equation}
Clearly, $Sx \in \ell_1$.   $S$ can be represented in matrix form by
%:   \eqref{S2}
\begin{equation}\label{S2}
\left(\begin{matrix}
(Sx)_1\\(Sx)_2\\(Sx)_3\\(Sx)_4\\(Sx)_5\\\vdots
\end{matrix}
\right)
= 
\left(\begin{matrix}
-1&-2&-2&-2&-2&\cdots\\
0&-1&-2&-2&-2&\cdots\\
0&0&-1&-2&-2&\cdots\\
0&0&0&-1&-2&\cdots\\
0&0&0&0&-1&\cdots\\
\vdots&\vdots&\vdots&\vdots&\vdots&\ddots
\end{matrix}
\right)
\left(\begin{matrix}
x_1\\x_2\\x_3\\x_4\\x_5\\\vdots
\end{matrix}
\right).
\end{equation}
If $x \in K$ then $t(x)_1 = 0$ and so, for all $k \ge 1$,
%:  \eqref{S3}
\begin{equation}\label{S3}
\left.\begin{aligned}
\tsum_{j = 1}^{k}x_j(Sx)_j
&= \tsum_{j = 1}^{k}\big(t(x)_{j} - t(x)_{j + 1}\big)\big(-t(x)_{j} - t(x)_{j + 1}\big)\\
&= t(x)_{k + 1}^2 - t(x)_{1}^2 = t(x)_{k + 1}^2.
\end{aligned}
\right\}
\end{equation}
Letting $k \to \infty$ in \eqref{S3}, for all $x \in K$,
%:   \eqref{S4}
\begin{equation}\label{S4}
\bra{x}{Sx} = \limn_{k \to \infty}\tsum_{j = 1}^{k}x_j(Sx)_j = \limn_{k \to \infty}t(x)_{k + 1}^2 = 0.
\end{equation}
If $x \in c_0 \setminus K$, we define $Sx := \emptyset$ .   Thus $S\colon\ c_0 \toto \ell_1$ is at most single--valued, linear and skew and $D(S) = K$.
\par
If $i \ge 1$, write $\e{i}$ for the sequence $(0,\dots, 0,1,0,0,\dots)$, with the 1 in the $i$th place.
\end{definition}    
%:   \Lem{SMAXlem}
\begin{lemma}\label{SMAXlem}
Let $j \ge 1$.  Then
\begin{equation*}
\e{j} - \e{j + 1} \in K\quand S\big(\e{j} - \e{j + 1}\big) = \e{j} + \e{j + 1}.
\end{equation*}
In other words,
\begin{equation*}
(\e{j} - \e{j + 1},\e{j} + \e{j + 1}) \in G(S).
\end{equation*}
\end{lemma}
\begin{proof}
Let $x := \e{j} - \e{j + 1}$.   It is easily seen that $t =  - \e{j + 1}$.   So, for all $p \ge 1$, 
\begin{equation*}
(Sx)_p = \e{j + 1}_p + \e{j + 1}_{p + 1} =
\begin{cases}0 + 0 = 0&\hbox{if }p < j;\\
0 + 1 = 1&\hbox{if }p = j;\\
1 + 0 = 1&\hbox{if }p = j + 1;\\
0 + 0 = 0&\hbox{if }p > j + 1.
\end{cases}
\end{equation*}
This gives the desired result.   Alternatively, we can simply subtract the $(j + 1)$st column from the $j$th column of the matrix in \eqref{S2}.   
\end{proof}
%:   \Thm{SMAXthm}
\begin{theorem}\label{SMAXthm}
$S$ is skew and maximally monotone but not quasidense.
\end{theorem}
\begin{proof}
From \eqref{S4}, $S$ is skew.   Now let $(x,x^*) \in c_0 \times \ell_1$ and,
%:  \eqref{S5}
\begin{equation}\label{S5}
\all\ z \in K,\ \bra{z - x}{Sz - x^*} \ge 0.
\end{equation}
From \eqref{S4}, $\bra{z}{Sz} = 0$, and so \eqref{S5} reduces to $\bra{x}{Sz} + \bra{z}{x^*} \le \bra{x}{x^*}$.   Since $K$ is a vector space, this implies that
%:  \eqref{S6}
\begin{equation}\label{S6}
\bra{x}{x^*} \ge 0\hbox{\quad and,\quad }\all\ z \in K,\ \bra{x}{Sz} = -\bra{z}{x^*}.
\end{equation}
\Lem{SMAXlem} and \eqref{S6} imply that, for all $j \ge 1$, 
\begin{equation*}
x_j + x_{j +1} = \Bra{x}{\e{j} + \e{j + 1}} = -\Bra{\e{j} - \e{j + 1}}{x^*} =  - x^*_j + x^*_{j + 1}.
\end{equation*}
Consequently, for all $n \ge 1$,
\begin{equation*}
x_{1} + x_{2} + \cdots + x_{2n - 1} + x_{2n} = -x_{1}^* + x_{2} ^*+ \cdots - x_{2n - 1}^* + x_{2n}^*.
\end{equation*}
Adding $x_{2n + 1}$ to both sides of this equation,
\begin{equation*}
x_{1} + x_{2} + \cdots + x_{2n - 1} + x_{2n} + x_{2n + 1} = -x_{1}^* + x_{2} ^*+ \cdots - x_{2n - 1}^* + x_{2n}^* + x_{2n + 1}.
\end{equation*}
Using the fact that $x \in c_0$, $x^* \in \ell_1$ and a simple interleaving argument, we see that $\tsum_{i = 1}^{\infty}x_i = \tsum_{i = 1}^{\infty}(-1)^{i}x^*_i$.   Since we now know that $\tsum_{i = 1}^{\infty}x_i$ is convergent, we can use the notation of \Def{Qdef}.   Thus 
%:  \eqref{S7}
\begin{equation}\label{S7}
t(x)_1 = \tsum_{i = 1}^{\infty}(-1)^{i}x^*_i.
\end{equation}
Let $j \ge 1$.   Using the same argument as above but starting the summation at $i = j$ instead of $i = 1$, $t(x)_{j} = -x_{j}^* + x_{j + 1} ^*-+ \cdots$.   Replacing $j$ by $j + 1$, $t(x)_{j + 1} = -x_{j + 1}^* + x_{j + 2} ^*-+ \cdots$ and so, by addition, $x_{j}^* = -t(x)_{j} - t(x)_{j + 1}$.   From \eqref{S1}, $x_j^* = (Sx)_j$.   So $Sx = x^* \in \ell_1$.   Furthermore,  
\begin{equation*}
\bra{x}{x^*} = \tsum_{j = 1}^\infty x_jx^*_j = \tsum_{j = 1}^\infty\big(t(x)_{j} - t(x)_{j + 1}\big)\big(-t(x)_{j} - t(x)_{j + 1}\big) = - t(x)_{1}^2,
\end{equation*} 
and \eqref{S6} now gives $t(x)_{1} = 0$.   Thus, from \eqref{S0}, $x \in K$ and $(x,x^*) \in G(S)$.   This completes the proof of the maximal monotonicity of $S$.
\par
We now prove that $S$ is not quasidense.   To this end, let $x \in K$.   Then, from \eqref{S1}, $\tsum_{j = 1}^{\infty}(-1)^{j}(Sx)_j = \big(t(x)_{1} + t(x)_{2}\big) - \big(t(x)_{2} + t(x)_{3}\big) + \big(t(x)_{3} + t(x)_{4}\big) \cdots = t(x)_1 = 0$, thus
\begin{equation*}
(Sx)_1 = \tsum_{j = 2}^{\infty}(-1)^{j}(Sx)_j,\hbox{ from which }|(Sx)_1| \le \tsum_{j = 2}^{\infty}|(Sx)_j|.
\end{equation*}
Thus $2|(Sx)_1| \le \tsum_{j = 1}^{\infty}|(Sx)_j| = \|Sx\|_1$.  Since $(Sx)_1 = -t(x)_1 - t(x)_2 =\break t(x)_1 - t(x)_2 = x_1$, $\|Sx\|_1^2 \ge 4x_1^2$.   From \eqref{S4}, $\bra{x - \e{1}}{Sx} = \bra{x}{Sx} - (Sx)_1 = -x_1$, and so
\begin{align*}
r_L\big((x,Sx) - (\e{1},0)\big) &= \half\|x - \e{1}\|_\infty^2 + \half\|Sx\|_1^2 + \bra{x - \e{1}}{Sx}\\
&\ge \half(x_1 - 1)^2 + 2x_1^2  - x_1 \ge \ts\frac{5}{2}x_1^2 - 2x_1 + \half\\
&= \ts\frac{5}{2}(x_1 - \frac{2}{5})^2 + \frac{1}{10} \ge \frac{1}{10}.
\end{align*}
This completes the proof that $S$ is not quasidense.
\end{proof}
\begin{remark}
As we observed above, $D(S) = K \ne c_0$.  On the other hand, the tail operator, $T$, defined in \Ex{TAILex} has full domain.   This leads to the following problem.
\end{remark}
\begin{problem}
Is every maximally monotone multifunction $T\colon\ c_0 \toto \ell_1$ such that $D(T) = c_0$ quasidense?
\end{problem}
It is natural to ask whether \Thm{TSthm}(b) can be used to establish the\break maximal monotonicity of $S$ in \Thm{SMAXthm}.   \Thm{TNOTMAXthm} below shows that this is impossible.
%:   \Lem{SOlem}
\begin{lemma}\label{SOlem}
Let $S\colon\ c_0 \toto \ell_1$ be as in \Def{Qdef}, $(x,x^*) \in G(S)$ and $\omega\dbs := (-1,1,-1,1,-1,\dots) \in \ell_\infty$.   Then $\bra{x^*}{\omega\dbs} = 0$.
\end{lemma}
\begin{proof}
$(x,x^*) \in c_0 \times \ell_1$ and, from \eqref{S0} and \eqref{S1}, $\tsum_{i = 1}^\infty x_i = 0$ and, for all $j \ge 1$, $x^*_j := -t(x)_{j} - t(x)_{j + 1}$.   Thus, for all $j \ge 1$,
\begin{equation*}
- x^*_j + x^*_{j + 1} = t(x)_{j} + t(x)_{j + 1}-t(x)_{j + 1} - t(x)_{j + 2} = x_j + x_{j + 1}.
\end{equation*}
Consequently,
\begin{align*}
\bra{x^*}{\omega\dbs} &= (- x^*_1 + x^*_2) + (- x^*_3 + x^*_4) + (- x^*_5 + x^*_6) + \dots\\
&= (x_1 + x_2) + (x_3 + x_4) + (x_5 + x_6) + \dots = \tsum_{i = 1}^\infty x_i = 0.
\end{align*}
This gives the desired result.
\end{proof}
%:   \Thm{TNOTMAXthm}
\begin{theorem}\label{TNOTMAXthm}
Let $S$ be as in {\em\Thm{SMAXthm}}, $T\colon\ \ell_1 \toto \ell_\infty$, $R(T) \subset \wh{c_0}$ and $S = G^{-1}L^{-1}G(T)$.   Then $T$ is not maximally monotone.
\end{theorem}
\begin{proof}
Let $(y^*,y\dbs) \in G(T)$.  From the proof of \Thm{TSthm}(a), there exists $(y,y^*) \in G(S)$ such that $\wh y = y\dbs$, and \Lem{SOlem} implies that $\bra{y^*}{\omega\dbs} = 0$.
From \eqref{S4}, $\bra{y^*}{y\dbs} = \bra{y}{y^*} = \bra{y}{Sy} = 0$, from which $\bra{y^* - 0}{y\dbs - \omega\dbs} =\break \bra{y^*}{y\dbs} -  \bra{y^*}{\omega\dbs} = 0$.   Thus $(0,\omega\dbs)$ is monotonically related to $G(T)$.\break   However, $\omega\dbs \not\in \wh{c_0} \supset R(T)$, and so $(0,\omega\dbs) \not\in G(T)$.
This completes the proof of \Thm{TNOTMAXthm}.
\end{proof}
%:\Sec{NONQDEXTsec}
\section{The Bueno--Svaiter construction}\label{NONQDEXTsec}
In \Ex{TAILex}, we gave an example of a quasidense maximally monotone\break multifunction with a non-quasidense Fitzpatrick extension.   In this section, we give a construction, due to Bueno and Svaiter, that produces another example of a similar phenonemon. \Def{Kdef} is patterned after Bueno, \cite[Theorem 2.7, pp.\ 13--14]{BUENO}.   It would be interesting to find a scheme that includes both the example of \Ex{TAILex}, and also examples of the kind considered in this section.   
%:   \Def{Kdef}
\begin{definition}\label{Kdef}
Let $E$ be a Banach space and $e\dbs \in E\dbs \setminus \wh E$. We define $k\colon\ E^* \to \RR$ by $k(y^*) = \bra{y^*}{e\dbs}^2$.   $k$ is a convex, continuous function on $E^*$.    Let $T\colon\ E^* \to E\dbs$ be a linear map and $R(T) \subset \wh E$.   Suppose that
%:  \eqref{GEN2}
\begin{equation}\label{GEN2}
\all\ x^* \in E^*,\ \bra{x^*}{Tx^*} = k(x^*) \ge 0.
\end{equation}
\end{definition}
In what follows, ``$\lin$'' stands for ``linear hull of''.
%:   \Lem{LMlem}
\begin{lemma}\label{LMlem}
$\dom\,k^* = \lin\{e\dbs\}$ and, for all $\mu \in \RR$, $k^*(2\mu e\dbs) = \mu^2$.
\end{lemma}
\begin{proof}
If $z\dbs \not\in \lin\{e\dbs\}$ then, from a well known algebraic result, there\break exists $z^* \in E^*$ so that $\bra{z^*}{e\dbs} = 0$ but $\bra{z^*}{z\dbs} \ne 0$.   Thus, for all $\lambda \in \RR$,\break $k^*(z\dbs) \ge \bra{\lambda z^*}{z\dbs} - \bra{\lambda z^*}{e\dbs}^2 = \lambda \bra{z^*}{z\dbs}$, and by taking $\lambda$ large and of the appropriate sign, $k^*(z\dbs) = \infty$.   Thus $\dom\,k^* \subset \lin\{e\dbs\}$.   If now $\mu \in \RR$ then $k^*(2\mu e\dbs) = \sup_{y^* \in E^*}\big[2\mu\bra{y^*}{e\dbs} - \bra{y^*}{e\dbs}^2\big]$.   Since $e\dbs \ne 0$, as $y^*$ runs through $E^*$, $\bra{y^*}{e\dbs}$ runs through $\RR$, and so (by elementary calaculus or\break completing the square) $k^*(2\mu e\dbs) = \sup_{\lambda \in \RR}\big[2\mu\lambda - \lambda^2\big] = \mu^2$.
\end{proof}
%:   \Thm{VWthm}
\begin{theorem}\label{VWthm}
$T$ is not quasidense.
\end{theorem}
\begin{proof}
We start off by proving that
%:   \eqref{W1}
\begin{equation}\label{W1}
\hbox{If }z\trs \in E\trs,\ \Bra{\wh E}{z\trs} = \{0\}\hbox{ and }\lambda \in \RR\hbox{ then }\theta_T(e\dbs,\lambda z\trs) = \fourth. 
\end{equation}
To this end, let $z\trs$ and $\lambda$ be as in \eqref{W1}.   From \eqref{GEN2} and the definition of $T$, for all $x^* \in E^*$, $\bra{x^*}{Tx^*} = k(x^*)$, and \eqref{THLONG} and \Lem{LMlem} give
\begin{align*}
\theta_{T}(e\dbs,\lambda z\trs) &= \supn_{x^* \in E^*}\big[\bra{x^*}{e\dbs} + \lambda\bra{Tx^*}{z\trs} - \bra{x^*}{Tx^*}\big]\\ 
&= \supn_{x^* \in E^*}\big[\bra{x^*}{e\dbs} + 0 - k(x^*)\big] = k^*(e\dbs) = \fourth. 
\end{align*}
This completes the proof of \eqref{W1}.   If $T$ were quasidense then, from \eqref{W1} and \Cor{THAcor}, if $z\trs \in E\trs$ and $\Bra{\wh E}{z\trs} = \{0\}$ then, for all $\lambda \in  \RR$, 
\begin{equation*}
\fourth = \theta_{T}(e\dbs,\lambda z\trs) \ge \bra{e\dbs}{\lambda z\trs} = \lambda\bra{e\dbs}{z\trs}.
\end{equation*}
Letting $\lambda \to \pm \infty$, $\bra{e\dbs}{z\trs} = 0$.  So we would have $\bra{e\dbs}{z\trs} = 0$ whenever $\Bra{\wh E}{z\trs} = \{0\}$.  Since $\wh E$ is a closed subspace of $E\dbs$, it would follow that $e\dbs \in \wh E$, violating the assumption in \Def{Kdef}.   
\end{proof}
%:   \Thm{SFTthm}
\begin{theorem}\label{SFTthm}
Let $S = G^{-1}L^{-1}G(T)$ {\em(see \Thm{TSthm})}.   Then:
\par\noindent
{\rm(a)}\enspace $S$ is maximally monotone and quasidense, and $S^\F = T$.
\par\noindent
{\rm(b)}\enspace $S^\F$ is maximally monotone but not quasidense. 
\end{theorem}
\begin{proof}
(a) is immediate from \Def{Kdef} and \Thm{LINVWthm}, and (b) is\break immediate from \Thm{AFMAXthm}, (a) and \Thm{VWthm}.
\end{proof}
For the rest of this section, we shall consider some of the more technical properties of $\theta_S$, with $S = G^{-1}L^{-1}G(T)$ as in \Thms{TSthm} and \ref{SFTthm}.
%:   \Lem{XYlem}
\begin{lemma}\label{XYlem}
For all $x^*,y^* \in E^*$, $\bra{y^*}{Tx^*} = \Bra{x^*}{2\bra{y^*}{e\dbs}e\dbs - Ty^*}$.
\end{lemma}
\begin{proof}
We have
\begin{align*}
\bra{y^*}{Tx^*} + \bra{x^*}{Ty^*} &= \half\bra{x^* + y^*}{Tx^* + Ty^*} - \half\bra{x^* - y^*}{Tx^* - Ty^*}\\
&= \half k(x^* + y^*) - \half k(x^* - y^*) = 2\bra{x^*}{e\dbs}\bra{y^*}{e\dbs}.
\end{align*}
The result follows easily from this.
\end{proof}
%:    \Thm{NPHWthm}
\begin{theorem}\label{NPHWthm}
Let $(y^*,y\dbs) \in E^* \times E\dbs$.   Then
%:   \eqref{NPHW1}
\begin{equation}\label{NPHW1}
(y^*,y\dbs) \in \dom\,\theta_S \iff 2\bra{y^*}{e\dbs}e\dbs - Ty^* + y\dbs \in \lin\{e\dbs\}.
\end{equation}
It follows that $\dom\,\theta_S$ is a linear subpace of $E^* \times E\dbs$.   Furthermore, for all $(y^*,y\dbs) \in \dom\,\theta_S$,  there exists a unique value of $\mu \in \RR$ such that 
%:   \eqref{NPHW3}
\begin{equation}\label{NPHW3}
2\bra{y^*}{e\dbs}e\dbs - Ty^* + y\dbs = 2\mu e\dbs,\hbox{ and then }\theta_S(y^*,y\dbs) = \mu^2.
\end{equation}
\end{theorem}
\begin{proof}
It follows from \eqref{THLONG} and \eqref{GEN2} that
%:   \eqref{NPHW6}
\begin{align*}
\theta_S(y^*,y\dbs) &=\supn_{x^* \in E^*}\big[\bra{y^*}{Tx^*} + \bra{x^*}{y\dbs} - k(x^*)\big].
\end{align*}
Thus, from \Lem{XYlem},
\begin{align*}
\theta_S(y^*,y\dbs) &=\supn_{x^* \in E^*}\big[\Bra{x^*}{2\bra{y^*}{e\dbs}e\dbs - Ty^* + y\dbs} - k(x^*)\big]\\
&= k^*\big(2\bra{y^*}{e\dbs}e\dbs - Ty^* + y\dbs\big).
\end{align*}
\eqref{NPHW1} now follows from \Lem{LMlem}. Since $e\dbs \ne 0$, for all $(y^*,y\dbs) \in \dom\,\theta_S$ there exists a unique $\mu \in \RR$ such that $2\bra{y^*}{e\dbs}e\dbs -Ty^* + y\dbs = 2\mu e\dbs$, and the rest of \eqref{NPHW3} follows from another application of \Lem{LMlem}.
\end{proof}
%:    \Cor{PHIVcor}
\begin{corollary}\label{PHIVcor}
$\dom\,\varphi_S = G(S)$ and $\theta_S = {\varphi_S}^*$ on $E^* \times E\dbs$.
\end{corollary}
\begin{proof}
Let $(x,x^*) \in \dom\,\varphi_S$.   From \eqref{TH2}, $(x^*, \wh x) \in \dom\,\theta_S$.   \Thm{NPHWthm} now gives a unique value of $\mu \in \RR$ such that  $2\bra{x^*}{e\dbs}e\dbs - Tx^* + \wh x  = 2\mu e\dbs$.   Thus $\wh E \ni \wh x - Tx^*  = 2\big(\mu - \bra{x^*}{e\dbs}\big)e\dbs$.   From \Def{Kdef}, $e\dbs \not\in \wh E$, and so $\mu - \bra{x^*}{e\dbs} = 0$, from which $\wh x - Tx^* = 0$.   It follows that $(x,x^*) \in G(S)$.   Thus $\dom\,\varphi_S \subset G(S)$.   The result now follows from \Lem{THlem}.
\end{proof}
Since $S$ is quasidense, it follows from \Thm{COINCthm} that
%:   \eqref{TELE2}
\begin{equation}\label{TELE2}
\dcoinc[{\theta_S}] = \dcoinc[{\varphi_S}^*] = \dcoinc[{\theta_S}^@].
\end{equation}
Of course, we know the first equality in \eqref{TELE2} from \Cor{PHIVcor}.
The second equality in \eqref{TELE2} leads naturally to the conjecture that ${\theta_S }^@ = {\varphi_S}^*$ on $E^* \times E\dbs$.   As we show in \Thm{THATthm} below, this conjecture fails in a spectacular way.   This raises the question of finding the exact value of $\dom\,{\theta_S}^@$.
%:   \Thm{THATthm}
\begin{theorem}\label{THATthm}
Since $e\dbs \ne 0$, there exists $y^* \in E^*$ so that $\bra{y^*}{e\dbs} = 1$.   Define $y\dbs \in E\dbs$ by $y\dbs := Ty^* - 2{e\dbs}$.   Let $\lambda \in \RR$.   Then
%:   \eqref{THAT1}
\begin{equation}\label{THAT1}
\theta_S(\lambda y^*,\lambda y\dbs) = 0,\hbox{ in particular, }{\varphi_S}^*(y^*,y\dbs) = \theta_S(y^*,y\dbs) = 0.
\end{equation}
but    
%:   \eqref{THAT2}
\begin{equation}\label{THAT2}
{\theta_S}^@(y^*,y\dbs) = \infty.
\end{equation}    
\end{theorem}
\begin{proof}
We note that $2\bra{\lambda y^*}{e\dbs}e\dbs - T\lambda y^* + \lambda y\dbs = \lambda(2e\dbs - Ty^* + y\dbs) = 0$, so \eqref{THAT1} follows from \eqref{NPHW3}.   Let $\lambda < 0$.   From \eqref{FAT} and \eqref{THAT1},
\begin{align*}
{\theta_S}^@(y^*,y\dbs)
&= \supn_{(x^*,x\dbs) \in E^* \times E\dbs}\big[\Bra{(x^*,x\dbs)}{(y\dbs,\wh{y^*})} - \theta_S(x^*,x\dbs)\big]\\
&\ge \bra{\lambda y^*}{y\dbs} + \bra{y^*}{\lambda y\dbs} - \theta_S(\lambda y^*,\lambda y\dbs) = 2\lambda\bra{y^*}{y\dbs}.
\end{align*}
However, $\bra{y^*}{y\dbs} = \bra{y^*}{Ty^* - 2{e\dbs}} = \bra{y^*}{Ty^*} - 2\bra{y^*}{{e\dbs}}$.   It now\break follows from \eqref{GEN2} that $\bra{y^*}{y\dbs} = \bra{y^*}{e\dbs}^2 - 2\bra{y^*}{{e\dbs}} = 1 - 2 = -1$, and so ${\theta_S}^@(y^*,y\dbs) \ge -2\lambda$, and we obtain \eqref{THAT2} by letting $\lambda \to -\infty$.
 \end{proof}
%:\Sec{SPECsec}
\section{A specific non--quasidense Fitzpatrick\\ extension}\label{SPECsec}
If $x^* \in \ell_1$ and $j \ge 1$, let $\tau_j :=\tsum_{i = j}^\infty x^*_i$.  Define the linear map $T\colon\ \ell_1 \to \ell_\infty$ by
%:   \eqref{TELE1}
\begin{equation}\label{TELE1}
\all\ j \ge 1,\  (Tx^*)_j = \tau_j + \tau_{j + 1}.
\end{equation}
Clearly $R(T) \subset \wh{c_0}$.   Let $e\dbs := (1,1,1,1,\dots) \in \ell_\infty \setminus \wh{c_0}$.
\begin{remark}
$T$ can be represented by
\begin{equation*}
\left(\begin{matrix}
(Tx^*)_1\\(Tx^*)_2\\(Tx^*)_3\\(Tx^*)_4\\(Tx^*)_5\\\vdots
\end{matrix}
\right)
=
\left(\begin{matrix}
1&2&2&2&2&\cdots\\
0&1&2&2&2&\cdots\\
0&0&1&2&2&\cdots\\
0&0&0&1&2&\cdots\\
0&0&0&0&1&\cdots\\
\vdots&\vdots&\vdots&\vdots&\vdots&\ddots
\end{matrix}
\right)
\left(\begin{matrix}
x_1^*\\x_2^*\\x_3^*\\x_4^*\\x_5^*\\\vdots
\end{matrix}
\right)
\end{equation*}
\end{remark}
%:   \Lem{Ulem}
\begin{lemma}\label{Ulem}
For all $x^* \in \ell_1$, $\bra{x^*}{Tx^*} =  \bra{x^*}{e\dbs}^2 \ge 0$.
\end{lemma}
\begin{proof}
Let $j \ge 1$.  Then $x^*_j(Tx^*)_j = (\tau_j + \tau_{j + 1})(\tau_j - \tau_{j + 1}) =  \tau_{j}^2 - \tau_{j + 1}^2$.   Since $x^* \in \ell_1$, $\lim_{k \to \infty}\tau_{k} = 0$.   Thus
\begin{align*}
\tsum_{j = 1}^\infty x^*_j(Tx^*)_j &= \limn_{k \to \infty}\tsum_{j = 1}^{k} x^*_j(Tx^*)_j= \limn_{k \to \infty}\tsum_{j = 1}^{k} (\tau_{j}^2 - \tau_{j + 1}^2) = \tau_{1}^2,
\end{align*}
as required.
\end{proof}
%:   \Thm{SPECTthm}
\begin{theorem}\label{SPECTthm}
Let $S = G^{-1}L^{-1}G(T)$.   Then $S$ is maximally monotone and quasidense, and $S^\F = T$ is maximally monotone but not quasidense.
\end{theorem}
\begin{proof}
This is immediate from \Lem{Ulem} and \Thm{SFTthm}.
\end{proof}
\begin{remark}
In this case we can give a direct proof that $T$ is not quasidense.   For all $x^* \in \ell_1$, $Tx^* \in \wh{c_0}$ and so $\|Tx^* - e\dbs\|_\infty \ge 1$, and $\bra{x^*}{Tx^* - e\dbs} = \bra{x^*}{Tx^*} - \bra{x^*}{e\dbs} = \bra{x^*}{e\dbs}^2 - \bra{x^*}{e\dbs}$.   Thus
\begin{align*}
r_L((x^*,Tx^*) - (0,e\dbs)) &= \half\|x^*\|_1^2 + \half\|Tx^* - e\dbs\|_\infty^2 + \bra{x^*}{Tx^* - e\dbs}\\
&\ge 0 + \half + \bra{x^*}{e\dbs}^2 - \bra{x^*}{e\dbs}\\
&= \fourth + \fourth(2\bra{x^*}{e\dbs} - 1)^2 \ge \fourth.
\end{align*}
Thus $T$ is not quasidense.  
\end{remark}
\begin{remark}
Define $x \in c_0$ by, for all $j \ge 1$, $x_j = (Tx^*)_j$.  Clearly, $x^* = Sx$.   Using the fact that $x \in c_0$, $x^* \in \ell_1$ and an interleaving argument similar to that used in \Thm{SMAXthm}, we see that, for all $j \ge 1$, $\tau_{j} = \tsum_{i = j}^\infty(-1)^{i + j}x_i$.   It follows that $S$ can be represented in matrix form on the appropriate domain by
\begin{equation*}
\left(\begin{matrix}
(Sx)_1\\(Sx)_2\\(Sx)_3\\(Sx)_4\\(Sx)_5\\\vdots
\end{matrix}
\right)
= \left(\begin{matrix}
1&-2&2&-2&2&\cdots\\
0&1&-2&2&-2&\cdots\\
0&0&1&-2&2&\cdots\\
0&0&0&1&-2&\cdots\\
0&0&0&0&1&\cdots\\
\vdots&\vdots&\vdots&\vdots&\vdots&\ddots
\end{matrix}
\right)
\left(\begin{matrix}
x_1\\x_2\\x_3\\x_4\\x_5\\\vdots
\end{matrix}
\right).
\end{equation*}
\end{remark}
%:bibliography

\end{document}